\documentclass[reqno]{amsart}
\title{Characters of level $1$ standard modules of $C_n^{(1)}$ as generating functions for generalised partitions}
\author{Jehanne Dousse and Isaac Konan}
\address{Université de Genève, Section de Mathématiques, 7-9 rue du Conseil-Général, 1205 Genève, Switzerland}
\address{Univ Lyon, CNRS, Université Claude Bernard Lyon 1, UMR5208, Institut Camille Jordan, F-69622 Villeurbanne, France}
\email{jehanne.dousse@unige.ch}
\email{konan@math.univ-lyon1.fr}

\usepackage{amssymb}
\usepackage{amsmath}
\usepackage{amsthm}
\usepackage{todonotes}
\usepackage{fullpage}
\usepackage{enumitem}
\usepackage{xcolor}
\usepackage{graphicx}
\usepackage{url}
\usepackage{textcomp}
\usepackage{dsfont}
\usepackage{tikz}
\usepackage[toc,page]{appendix}
\usepackage{float}

\allowdisplaybreaks

\definecolor{foge}{rgb}{0.1, 0.6, 0.1}

\newcommand{\N}{\mathbb{N}}

\numberwithin{equation}{section}
\numberwithin{figure}{section}

\newcommand{\Thm}[1]{Theorem \ref{#1}}

\newcommand{\Prp}[1]{Proposition \ref{#1}}

\newcommand{\ov}{\overline}

\newcommand{\Ll}{\Lambda}
\newcommand{\Pp}{\mathcal{P}}
\newcommand{\Ppp}{\Pp^{\gg}_{\co}}
\newcommand{\Z}{\mathbb{Z}}
\newcommand{\Od}{\mathcal{O}}
\newcommand{\F}{\mathcal{F}}
\newcommand{\Co}{\mathcal{C}}
\newcommand{\Ee}{\mathcal{E}}
\newcommand{\B}{\mathcal{B}}
\newcommand{\la}{\lambda}

\newcommand{\Rr}{\mathcal{R}}
\newcommand{\ep}{\epsilon}
\newcommand{\wt}{\overline{\mathrm{wt}}}
\newcommand{\Sc}{\mathcal{S}}

\newcommand{\ot}{\otimes}
\newcommand{\p}{\mathfrak p}
\newcommand{\co}{c_{g}}

\newcommand{\eit}{\tilde{e}_i}

\newcommand{\gf}{\mathfrak g}
\newcommand{\h}{\mathfrak h}
\DeclareMathOperator{\ch}{ch}
\DeclareMathOperator{\sgn}{sgn}
\DeclareMathOperator{\mult}{mult}

\newcommand{\I}{\{0,\dots,n\}}
\newcommand{\C}{\mathcal{C}}
\newcommand{\Pppp}{\Pp^{\gtrdot}_{\co}}

\numberwithin{equation}{section}
\newtheorem{theo}{Theorem}[section]
\newtheorem{prop}[theo]{Proposition}
\newtheorem{lem}[theo]{Lemma}
\newtheorem{cor}[theo]{Corollary}
\newtheorem{con}[theo]{Conjecture}
\newtheorem{rem}[theo]{Remark}

\theoremstyle{definition} \newtheorem{deff}[theo]{Definition}

\begin{document}
\begin{abstract}
We give a new simple formula for the energy function of a level $1$ perfect crystal of type $C_n^{(1)}$ introduced by Kang, Kashiwara and Misra. We use this to give several expressions for the characters of level $1$ standard modules as generating functions for different types of partitions. We then relate one of these formulas to the difference conditions in the conjectural partition identity of Capparelli, Meurman, Primc and Primc, and prove that their conjecture is true for all level $1$ standard modules. Finally, we propose a non-specialised generalisation of their conjecture.
\end{abstract}

\maketitle

\section{Introduction and statement of results}
This paper lies at the intersection between representation theory and combinatorics, the goal being to connect Rogers--Ramanujan type partition identities with characters of Lie algebra modules and crystal bases. The connection between these two fields originated with work of Lepowsky, Milne and Wilson~\cite{Le-Mi1,Lepowsky,Lepowsky2} who linked the Rogers--Ramanujan identities to characters of standard level $3$ modules of the affine Lie algebra $A_1^{(1)}$. But before going into detail, let us first recall some definitions and notations.

A partition of a positive integer $n$ is a non-increasing sequence of positive integers whose sum is $n$. For example, the partitions of $4$ are $(4), (3,1), (2,2), (2,1,1)$ and $(1,1,1,1)$. A partition identity is a theorem stating that for all $n$, the number of partitions of $n$ satisfying some conditions (often difference conditions between consecutive parts) equals the number of partitions of $n$ satisfying some other conditions (often congruence conditions on the parts).
Probably the most famous partition identities are those of Rogers--Ramanujan \cite{RogersRamanujan}. They were originally stated as $q$-series identities:
\begin{theo}[The Rogers--Ramanujan identities]
  \label{th:RR}
  Let $i=0$ or $1$. Then
  $$
  \sum_{n \geq 0} \frac{q^{n^2+ (1-i)n}}{(q;q)_n} = \frac{1}{(q^{2-i};q^5)_{\infty}(q^{3+i};q^5)_{\infty}}.
  $$
\end{theo}
\noindent Here the $q$-Pochhammer symbol $(a;q)_n$ is defined  for $n \in \mathbb{N} \cup \{\infty\}$ as
\begin{align*}
  (a;q)_n &:= \prod_{k=0}^{n-1} (1-aq^k).
\end{align*}
These identities were then interpreted as partition identities by MacMahon \cite{MacMahon} and Schur \cite{SchurRR}:
\begin{theo}[Rogers--Ramanujan identities, combinatorial version]
  \label{th:RRcomb}
  Let $i=0$ or $1$. For every natural number $n$, the number of partitions of $n$ such that the difference between two consecutive parts is at least $2$ and the part $1$ appears at most $i$ times is equal to the number of partitions of $n$ into parts congruent to $\pm (2-i) \mod 5.$
\end{theo}

Any partition $\lambda$ can be also be described in terms of frequencies, i.e. by a sequence $(f_i)_{i\geq 1}$, where $f_i$ is the number of parts of $\lambda$ which are equal to $i$. For example, the partition $(4,3,3,3,1,1)$ corresponds to the frequency sequence $(2,0,3,1,0,0,0,\dots)$. In the Rogers--Ramanujan identities, partitions with difference at least $2$ between consecutive parts can equivalently be seen as partitions whose frequency sequence satisfies $f_i + f_{i+1} \leq 1$ for all $i$. It is however not easy in general (or even possible) to find a correspondence between frequency conditions and difference conditions. For the Rogers--Ramanujan identities, it is the version with frequency conditions that Lepowsky and Wilson used to make the connection with characters of affine Lie algebras.

\medskip
We do not recall all the definitions related to Kac--Moody Lie algebras and crystals. The interested reader can find them in classical books like~\cite{HK,Kac}, or in our first paper on the subject \cite{DK19-2}. However we fix some notations.

Let $\gf$ be an affine Kac--Moody Lie algebra with Cartan subalgebra $\h$. Let $P^{+}$ denote the set of dominant integral weights and let $L(\lambda)$ be the irreducible highest weight $\gf$-module of highest weight $\lambda \in P^+$ (also called the standard module with highest weight $\lambda$). Let $\alpha_0, \dots , \alpha_n$ and $\Lambda_0, \dots, \Lambda_n$ be the simple roots and fundamental weights, respectively. Let $\delta = d_0 \alpha_0 + d_1 \alpha_1 + \cdots + d_{n-1} \alpha_{n-1}$ be the null root.

Then the character of $L(\lambda)$ is defined as
$$
\ch L(\lambda) = \sum_{\mu\in \h^*} \dim L(\lambda)_{\mu} \cdot e^{\mu},
$$
where $e$ is a formal exponential satisfying $e^{\mu + \mu'} = e^{\mu}e^{\mu'}$, and $L(\lambda)_{\mu}$ is the weight space of weight $\mu$ in the weight-space decomposition of $L(\lambda)$. 
The Weyl--Kac character formula \cite[Proposition 10.10]{Kac} expresses the character as follows:
\begin{align*} \label{eq:WeylKac}
  \ch L(\lambda) = \frac{\sum_{w \in W} \sgn(w) e^{w(\lambda+\rho)-\rho}}{\prod_{\alpha \in \Delta^{+}} (1-e^{-\alpha})^{\mult \alpha}},
\end{align*}
where $\Delta^{+}$ is the set of positive roots, $W$ is the Weyl group and $\rho$ is the Weyl vector.

Replacing $e^{-\alpha_i}$ by $q$ for all $i \in \I$  in a character formula is called performing the \textit{principal specialisation}. We denote the principal specialisation by the operator $\mathds{1}$ (see \eqref{eq:princchar} for an example).
Lepowsky and Milne \cite{Le-Mi1,Le-Mi2} first noted that, once multiplied with the ``fudge factor'' $1/(q;q^2)_{\infty}$,  the product side of the Rogers--Ramanujan identities becomes equal to the principal specialisation of the Weyl--Kac character formula for the level $3$ standard modules $3 \Lambda_0$ and $2 \Lambda_0+\Lambda_1$ of the affine Lie algebra $A_1^{(1)}$.
To establish the equality with the sum side of the Rogers--Ramanujan identities, Lepowsky and Wilson \cite{Lepowsky} constructed a basis of the same level $3$ standard modules using the theory of vertex operator algebras. This basis corresponds in some sense to partitions in which some patterns, equivalent to the frequency conditions $f_i + f_{i+1} \leq 1$, are forbidden.

\medskip

More generally, performing the principal specialisation in the Weyl--Kac character formula for any standard module of any affine Kac-Moody gives a product of $q$-Pochhammer symbols, see for example \cite{LepLectures}.

In particular, in type $C_n^{(1)}$, there is a precise expression for $\mathds{1}\left(e^{-k_0\Lambda_0-\cdots-k_n\Lambda_n}\chi(L(k_0\Lambda_0+\cdots+k_n\Lambda_n))\right)$ as an infinite product. 
Let $s$ be a non-negative integer, and let $x_0,\ldots,x_s$ be positive integers. Define the set
$$D(x_0,\ldots,x_s)=\{x_0+\cdots+ x_j: 0\leq j\leq s\}\sqcup \{x_0+\cdots+ x_{j-1}+2x_{j}+\cdots +2x_s: 1\leq j\leq s\},$$
and the multiset 
$$\Delta(x_0,\ldots,x_s) = D(x_0,\ldots,x_s)\sqcup D(x_1,\ldots,x_s)\sqcup \cdots \sqcup D(x_s).$$
Then
\begin{equation}
\label{eq:princchar}
\begin{aligned}
&\mathds{1}\left(e^{-k_0\Lambda_0-\cdots-k_n\Lambda_n}\ch(L(k_0\Lambda_0+\cdots+k_n\Lambda_n))\right) \\
&\quad= \frac{\prod_{a\in \{2n+2k+2\}^n \sqcup D(k_0+1,\ldots,k_n+1);b\in \Delta(k_1+1,\ldots,k_n+1); j=a,b,2n+2k+2-b}(q^j;q^{2n+2k+2})_\infty}{(q;q^2)_\infty(q;q)_\infty^n},
\end{aligned}
\end{equation}
where $k=k_0+\cdots+k_n$.

Lepowsky and Wilson's approach, applied to other affine Kac--Moody Lie algebras or at other levels, gives rise to other sum/product identities, also called Rogers--Ramanujan type partition identities, see e.g. \cite{Capparelli,Meurman2,Nandi,PrimcSikic,PrimcSikic2,Siladic}. However, while it is relatively easy to obtain an infinite product by using the principal specialisation in the Weyl--Kac character formula, it is much harder to find a basis of the modules considered. Hence, some of the aforementioned partition identities were only conjectured, not proved, using representation theory. But a combinatorial proof of the identity automatically proves that the conjectured basis is indeed a basis, by equality of cardinalities. Capparelli's identity \cite{Capparelli} was first proved combinatorially by Andrews \cite{Andrewscap}, then refined by Alladi, Andrews and Gordon in \cite{AllAndGor} using the method of weighted words, and finally proved by Capparelli \cite{Capparelli2}. Siladi\'c's identity was already proved by Siladi\'c himself \cite{Siladic} using purely representation theoretic techniques, but was then refined combinatorially by the first author \cite{Doussesil2}, and latter generalised by the second author \cite{Konan_Sil} and linked to statistical mechanics \cite{Konan_ww1}. 
Nandi's identity \cite{Nandi} was conjectured using vertex operator algebras and proved combinatorially by Takigiku and Tsuchioka \cite{Takigiku2019APO}.

\medskip
Using the vertex operator algebras approach, Primc and {\v{S}}iki\'c \cite{PrimcSikic} proved an identity for characters of $L(\Lambda_0)$ in type $C_n^{(1)}$ for all $n \geq 1$, and conjectured a generalisation for $L(k\Lambda_0)$ in type $C_n^{(1)}$ for all $n,k \geq 1$ in \cite{PrimcSikic2}. 
This conjecture describes the characters as generating functions for partitions with frequency conditions on so-called ``cascades'', which are downward paths in certain tables of integers, and the principal specialisation gives a sum/product identity.

Using computer algebra experimentations, Capparelli, Meurman, Primc and Primc \cite{CMPP} recently generalised Primc and {\v{S}}iki\'c's conjecture to all standard modules of level $k$ in type $C_n^{(1)}$ for all $n,k \geq 1$. We give some notation to be able to state their conjecture.
Let $\mathcal{N}_{2n+1}$ be the array of integers with $2n+1$ rows represented in \eqref{intarray}:
\begin{equation}\label{intarray}
\begin{matrix}
0&& 1& &3& & 5& & 7&    \\
 & 0&&2& &4 & &6 & &8  \\
 & &&& &\vdots& &  & & \\
 0&& 1& &3& & 5& & 7&   \\
 & 0&&2& &4 & &6 & & 8  \\
0&& 1& &3& & 5& & 7&    \\
\end{matrix}\quad\dots.
\end{equation}
Here an integer $l$ in row $i$ is considered to be different from the integer $l$ in row $j$ for $j \neq i$ (or equivalently, we can consider that each row is coloured with a different colour).
Any partition with parts in $\mathcal{N}_{2n+1}$ can be represented by its frequencies in the array $\mathcal{F}^{k_0,k_1, \dots , k_n}$ represented in \eqref{freqarray}:
\begin{equation}\label{freqarray}
\begin{matrix}
 k_{n}&& f_1& &f_ 3& & f_5& & f_7&    \\
 & 0&&f_ 2& &f_4 & &f_6 & & f_{8}  \\
 & &&& &\vdots& &  & & \\
 k_{1}&& f_1& &f_ 3& & f_5& & f_7&   \\
 & 0&&f_ 2& &f_4 & &f_6 & & f_{8}  \\
 k_{0}&& f_1& &f_ 3& & f_5& & f_7&    \\
\end{matrix}\quad\dots,
\end{equation}
where the zeros in the first column are set to have fictitious frequencies equal to $k_0,k_1, \dots , k_n$.

A \textit{downward path} $\mathcal{Z}$ in an array with $2n+1$ rows is a $(2n+1)$-tuple $(a_1,a_2,a_3,\dots,a_{2n+1})$ such that $a_i$ is in the $i$-th row for all $1 \leq i \leq 2n+1$ and $(a_i, a_{i+1})$ is a pair of two adjacent elements for all $1 \leq i \leq 2n$.
For example,
two downward paths in $\mathcal{F}^{2,0,0}$ are represented in red and blue below:
$$
\begin{matrix}
0&&{\color{red}f_{1}}& & f_{3}& & {\color{blue}f_{5}}& & f_{7}&   \\
& {\color{red}0}&& f_{2}& &f_{4} & &{\color{blue}f_{6}} & & f_{8}  \\
0&&{\color{red}f_{1}}& &f_{3} & &f_{5} & &{\color{blue}f_{7}} &  \\
& {\color{red}0}&&f_{2} & &f_{4} & &f_{6} & &{\color{blue}f_{8}}   \\
{\color{red}2}&&f_{1}& &f_{3} & &f_{5} & &{\color{blue}f_{7}}&
\end{matrix}\quad\dots.
$$
A partition with parts in $\mathcal{N}_{2n+1}$ is said to be $(k_0,k_1, \dots , k_n)$-\textit{admissible} if its frequencies in $\mathcal{F}^{k_0,k_1, \dots , k_n}$ satisfy
$$\sum_{m \in \mathcal{Z}} m \leq k_0+\cdots+k_n,$$ for all admissible paths $\mathcal{Z}$  in $\mathcal{F}^{k_0,k_1, \dots , k_n}$.

We can now state the Capparelli--Meurman--Primc--Primc (CMPP) conjecture.
\begin{con}[CMPP conjecture]
\label{con:CMPP}
Let $k_0,\ldots,k_n$ be non-negative integers. The principally specialized character
$\mathds{1}\left(e^{-k_0\Lambda_0-\cdots-k_n\Lambda_n}\ch(L(k_0\Lambda_0+\cdots+k_n\Lambda_n))\right)$
is the generating function for the number of $(k_0, k_1, \dots, k_n)$-admissible partitions with parts in $\mathcal{N}_{2n+1}$.
\end{con}

The case $k_0=k$, $k_1 = \cdots = k_n =0$ corresponds to the conjecture of Primc--{\v{S}}iki\'c.

Note that the type $C_n^{(1)}$ is defined for $n\geq 2$. However, by replacing in Conjecture \ref{con:CMPP} the principal specialization of the character by the product \eqref{eq:princchar}, it becomes possible to consider the case $n=1$ from a combinatorial viewpoint.
By doing so, we retrieve an identity of Meurman--Primc \cite{Meurman2} related to the character of the standard module $L(k_0\Lambda_0+k_1\Lambda_1)$ of type $A_1^{(1)}$.
\medskip

In this paper, we give several expressions for the characters of all level $1$ standard modules $\Lambda_0, \dots , \Lambda_n$ of $C_n^{(1)}$ in terms of generating functions for different types of generalised partitions. One of these generating functions, when performing the principal specialisation, becomes exactly the generating function for the admissible partitions of the CMPP conjecture.

Hence, our result implies that \textbf{the CMPP conjecture is true for all level $1$ standard modules $\Lambda_0, \dots , \Lambda_n$ of $C_n^{(1)}$, i.e. for the case $k_0+ \cdots + k_n =1$ and all $n$}.

\medskip We derive these character formulas in terms of partitions via the theory of perfect crystals. Again we refer the reader to \cite{HK} for all the relevant definitions. Intuitively, the crystal of a standard module is an oriented graph which encodes some representation theoretic information on the module. This theory was introduced by Kang, Kashiwara, Misra, Miwa, Nakashima, and Nakayashiki \cite{KMN2a,KMN2b}, who proved the so-called ``(KMN)$^2$ crystal base character formula'' \cite{KMN2a}, which expresses the character as a sum of formal exponentials indexed by so-called $\lambda$-paths, i.e. sequences of vertices in the crystal graph which are ultimately equal to the so-called ``ground state path"
$$
{\p}_\lambda = \bigl(g_k)_{k=0}^\infty =  \cdots \ot g_{k+1} \ot g_k \ot \cdots \ot g_1 \ot g_0.
$$

Primc \cite{Primc} was the first to use the theory of perfect crystals and the (KMN)$^2$ crystal base character formula to conjecture and prove new partition identities corresponding to level $1$ standard modules of $A_1^{(1)}$ and $A_2^{(1)}$. The $A_1^{(1)}$ conjecture was proved and refined combinatorially by the first author and Lovejoy \cite{DoussePrimc} and the refined version was generalised to $A_n^{(1)}$ for all $n$ by the two authors \cite{DK19,DK19-2}.

In \cite{DK19-2}, the authors introduced the theory of grounded partitions, which allows one to rewrite the character as a generating function for certain coloured partitions, when the ground state path is constant, i.e. of the form $ \cdots \ot g_0 \ot g_0 \ot g_0$.
They later generalised their theory to multi-grounded partitions in \cite{DKmulti} to treat all possible ground state paths and applied it to higher levels of $A_1^{(1)}$ together with Hardiman in \cite{DHK}, obtaining a companion to the Andrews--Gordon \cite{AndrewsGordon} and Meurman--Primc  \cite{Meurman2} identities. In this paper we only study perfect crystals with constant ground state paths, 
so the theory of grounded partitions of \cite{DK19-2}, which we now briefly recall, is sufficient for our purposes.

\begin{deff}
\label{def:generalisedcolouredpartition}
Let $\C$ be a set of colours, let $\Z_{\C} = \{k_c : k \in \Z, c \in \C\}$ be the corresponding set of coloured integers, and let $\succ$ be a binary relation defined on $\Z_{\C}$.
A \textit{generalised coloured partition} with relation $\succ$ is a finite sequence $(\pi_0,\dots,\pi_s)$ of coloured integers from $\Z_{\C} $, such that for all $i \in \{0, \dots, s-1\},$ $\pi_i\succ \pi_{i+1}.$
\end{deff}
When it is clear from the context, we sometimes write simply ``partitions" instead of ``generalised coloured partition".
Moreover, if $k_c$ is a coloured integer, then $c(k_c)$ (resp. $|k_c|$) denotes the colour $c$ (resp. size $k$) of $k_c$. Similarly, if $\lambda$ is a generalised coloured partition, $|\lambda|$ denotes the size of $\lambda$, i.e. the sum of the sizes of its parts. Moreover, we let $\ell(\lambda)$ denote the number of parts of $\lambda$ and $C(\lambda)$ be colour sequence of $\lambda$, i.e. the product of the colours of its parts.

Now we fix a particular colour $\co \in \C$ and define grounded partitions. 
\begin{deff}
\label{def:groundedpartitions}
A \textit{grounded partition} with ground $\co$ and relation $\succ$ is a non-empty generalised coloured partition $\pi=(\pi_0,\ldots,\pi_s)$ with relation $\succ$, such that $\pi_s = 0_{\co}$, and when $s>0$, $\pi_{s-1}\neq 0_{\co}.$

\noindent Let $\Pp^{\succ}_{c_g}$ denote the set of such partitions.
\end{deff}

We proved the following character formula as generating function for grounded partitions.

\begin{theo}[Dousse--Konan 2019]
\label{th:formchar}
Let $\gf$ be an affine Kac--Moody Lie algebra, let $\B$ be a perfect crystal of level $\ell$, let $\Ll$  be a dominant integral weight of level $\ell$ with constant ground state path $\cdots \ot g \ot g$. Let $H$ be an energy function on $\B \ot B$ such that that $ H(g \ot g)=0$.
Then, letting $\Co=\{c_b:b\in \B\}$ be the set of colours indexed by the vertices of $\B$, and setting $q=e^{-\frac{\delta}{d_0}}$ and $c_b=e^{\wt b}$ for all $b\in \B$, we have $\co=1$, and
\begin{align*}
\sum_{\pi\in \Pppp} C(\pi)q^{|\pi|} &= e^{-\Ll}\rm{ch}(L(\Ll)),\\
 \sum_{\pi\in \Ppp} C(\pi)q^{|\pi|} &= \frac{e^{-\Ll}\rm{ch}(L(\Ll))}{(q;q)_{\infty}},
\end{align*}
where
$\Pppp$ and $\Ppp$ are respectively the sets of grounded partitions with ground $c_g$ and relations $\gtrdot$ and $\gg$ defined by
\begin{equation*}\label{eq:egal}
 k_{c_b}\gtrdot k'_{c_{b'}} \text{ if and only if } k-k'= H(b'\ot b).
\end{equation*}
and
\begin{equation*}\label{eq:relation}
 k_{c_b}\gg k'_{c_{b'}} \text{ if and only if } k-k'\geq H(b'\ot b).
\end{equation*}
\end{theo}

Define the set
$$\overline{[n]}:= \{1, \dots , n, \overline{n} , \dots , \overline{1}\},$$
with the convention that for all $k \in \ov{[n]}$, $\overline{\overline{k}}=k$, and the order
\begin{equation}
\label{eq:orderbar}
1<2<\cdots <n-1<n < \overline{n}< \overline{n-1}<\cdots<\overline{2}<\overline{1}.
\end{equation}
In this paper, we study the level $1$ perfect crystal $\B$ of $C_n^{(1)}$ represented in Figure \ref{fig:crystalb}.

\begin{figure}[H]
\begin{center}
\begin{tikzpicture}[scale=0.8, every node/.style={scale=0.8}]

\foreach \x in {0,2,6,8,12,14}
\draw (\x-0.35,0-0.2)--(\x+0.35,0-0.2)--(\x+0.35,0+0.3)--(\x-0.35,0+0.3)--cycle;

\draw (0,0) node {$1,1$};
\draw (2,0) node {$1,2$};
\draw (6,0) node {$1,n$};
\draw (8,0) node {$1,\ov{n}$};
\draw (12,0) node {$1,\ov{2}$};
\draw (14,0) node {$1,\ov{1}$};

\foreach \x in {2,6,8,12,14}
\draw (\x-0.35,1.5-0.2)--(\x+0.35,1.5-0.2)--(\x+0.35,1.5+0.3)--(\x-0.35,1.5+0.3)--cycle;

\draw (2,1.5) node {$2,2$};
\draw (6,1.5) node {$2,n$};
\draw (8,1.5) node {$2,\ov{n}$};
\draw (12,1.5) node {$2,\ov{2}$};
\draw (14,1.5) node {$2,\ov{1}$};

\foreach \x in {6,8,12,14}
\draw (\x-0.35,4.5-0.2)--(\x+0.35,4.5-0.2)--(\x+0.35,4.5+0.3)--(\x-0.35,4.5+0.3)--cycle;

\draw (6,4.5) node {$n,n$};
\draw (8,4.5) node {$n,\ov{n}$};
\draw (12,4.5) node {$n,\ov{2}$};
\draw (14,4.5) node {$n,\ov{1}$};

\foreach \x in {8,12,14}
\draw (\x-0.35,6-0.2)--(\x+0.35,6-0.2)--(\x+0.35,6+0.3)--(\x-0.35,6+0.3)--cycle;

\draw (8,6) node {$\ov{n},\ov{n}$};
\draw (12,6) node {$\ov{n},\ov{2}$};
\draw (14,6) node {$\ov{n},\ov{1}$};

\foreach \x in {12,14}
\draw (\x-0.35,9-0.2)--(\x+0.35,9-0.2)--(\x+0.35,9+0.3)--(\x-0.35,9+0.3)--cycle;

\draw (12,9) node {$\ov{2},\ov{2}$};
\draw (14,9) node {$\ov{2},\ov{1}$};

\foreach \x in {0,14}
\draw (\x-0.35,10.5-0.2)--(\x+0.35,10.5-0.2)--(\x+0.35,10.5+0.3)--(\x-0.35,10.5+0.3)--cycle;

\draw (0,10.5) node {$\emptyset$};
\draw (14,10.5) node {$\ov{1},\ov{1}$};


\foreach \x in {4,10}
\foreach \y in {0,1.5}
\draw (\x,\y+0.05) node {$\dots$};

\foreach \x in {12,14}
\foreach \y in {3,7.5}
\draw (\x,\y+0.05) node {$\vdots$};

\foreach \x in {10}
\foreach \y in {4.5,6}
\draw (\x,\y+0.05) node {$\dots$};

\foreach \x in {6,8}
\foreach \y in {3}
\draw (\x,\y+0.05) node {$\vdots$};

\foreach \x in {0,...,6}
\draw[->] (2*\x+0.35,0+0.05)--(2*\x+1.65,0+0.05);
\foreach \y in {0,...,6}
\draw[->] (14,1.5*\y+0.3)--(14,+1.5*\y+1.3);

\foreach \x in {1,...,5}
\draw[->] (2*\x+0.35,1.5+0.05)--(2*\x+1.65,1.5+0.05);
\foreach \y in {1,...,5}
\draw[->] (12,1.5*\y+0.3)--(12,+1.5*\y+1.3);

\foreach \x in {3,5,6}
\draw[->] (2*\x+0.35,4.5+0.05)--(2*\x+1.65,4.5+0.05);
\foreach \y in {0,1,3}
\draw[->] (8,1.5*\y+0.3)--(8,+1.5*\y+1.3);

\foreach \x in {4,...,6}
\draw[->] (2*\x+0.35,6+0.05)--(2*\x+1.65,6+0.05);
\foreach \y in {0,...,2}
\draw[->] (6,1.5*\y+0.3)--(6,+1.5*\y+1.3);

\foreach \x in {6}
\draw[->] (2*\x+0.35,9+0.05)--(2*\x+1.65,9+0.05);
\foreach \y in {0}
\draw[->] (2,1.5*\y+0.3)--(2,+1.5*\y+1.3);

\draw[dotted,->] (0,10.3)--(0,0.3);
\draw[dotted,->] (13.65,10.55)--(0.35,10.55);

\draw[dotted,<-] (2,-0.2)--(2,-1.7)--(14.85,-1.7)--(14.85,1.55)--(14.35,1.55);
\draw[dotted,<-] (6,-0.2)--(6,-1.325)--(15.35,-1.325)--(15.35,4.55)--(14.35,4.55);
\draw[dotted,<-] (8,-0.2)--(8,-0.95)--(15.85,-0.95)--(15.85,6.05)--(14.35,6.05);
\draw[dotted,<-] (12,-0.2)--(12,-0.575)--(16.35,-0.575)--(16.35,9.05)--(14.35,9.05);

\draw (0,5.22) node[left] {\begin{footnotesize}$0$\end{footnotesize}};
\draw (7,10.7) node {\begin{footnotesize}$0$\end{footnotesize}};
\foreach \x in {0,...,3}
\draw (14.85+0.5*\x,0.72) node[right] {\begin{footnotesize}$0$\end{footnotesize}};

\foreach \x in {2,6,8,14}
\draw (\x,0.72) node[right] {\begin{footnotesize}$1$\end{footnotesize}};
\draw (14,9.72) node[right] {\begin{footnotesize}$1$\end{footnotesize}};

\foreach \y in {0,4.5,6,9}
\draw (13,\y-0.1) node {\begin{footnotesize}$1$\end{footnotesize}};
\draw (1,-0.1) node {\begin{footnotesize}$1$\end{footnotesize}};

\foreach \x in {6,8,12,14}
\draw (\x,2.22) node[right] {\begin{footnotesize}$2$\end{footnotesize}};
\foreach \x in {12,14}
\draw (\x,8.22) node[right] {\begin{footnotesize}$2$\end{footnotesize}};

\foreach \y in {0,1.5,4.5,6}
\draw (11,\y-0.1) node {\begin{footnotesize}$2$\end{footnotesize}};
\foreach \y in {0,1.5}
\draw (3,\y-0.1) node {\begin{footnotesize}$2$\end{footnotesize}};

\foreach \x in {6,12,14}
\draw (\x,3.72) node[right] {\begin{footnotesize}$n-1$\end{footnotesize}};
\foreach \x in {12,14}
\draw (\x,6.72) node[right] {\begin{footnotesize}$n-1$\end{footnotesize}};

\foreach \y in {0,1.5,6}
\draw (9,\y-0.1) node {\begin{footnotesize}$n-1$\end{footnotesize}};
\foreach \y in {0,1.5}
\draw (5,\y-0.1) node {\begin{footnotesize}$n-1$\end{footnotesize}};

\foreach \x in {8,12,14}
\draw (\x,5.22) node[right] {\begin{footnotesize}$n$\end{footnotesize}};
\foreach \y in {0,1.5,4.5}
\draw (7,\y-0.1) node {\begin{footnotesize}$n$\end{footnotesize}};

\end{tikzpicture}

\caption{Crystal graph of $\B$}
\label{fig:crystalb}
\end{center}
\end{figure}

Our first main result is a simple formula for the energy function on $\B \ot \B$, in the spirit of the formula which we gave for the energy function of a level $1$ crystal of $A_n^{(1)}$ in \cite{DK19,DK19-2}.

\begin{theo}
\label{th:energy}
The energy function $H$ on $\B \ot \B$ such that $H( \emptyset \ot \emptyset)=0$ is given, for all $x \leq y, x' \leq y' \in \ov{[n]}$, by the formulas
$$
\begin{cases}
H(\emptyset\ot \emptyset)=0,\\
H(\emptyset\ot (x,y))=H((x,y)\ot\emptyset )=1,
\end{cases}
$$
and
$$
H((x,y)\ot (x',y')) = \begin{cases}
\chi(x\geq x')+\chi(y\geq y')-\chi(y\geq y'>x \geq x') \ \ \text{if} \ \ \overline{y'}\neq x, \\
\chi(x>x')+\chi(y> y')-\chi(y> y'>x > x') \ \ \text{if} \ \ \overline{y'}= x,
\end{cases}
$$
where $\chi (prop)=1$ if the proposition $prop$ is true and $0$ otherwise.
\end{theo}

Combining Theorems \ref{th:formchar} and \ref{th:energy}, we have an expression for the character as generating functions for grounded partitions which satisfy difference conditions given by the energy $H$.

\begin{theo}
\label{th:charwithenergy}
Let $b_0=\emptyset$ and $b_i=(i,\overline{i})$ for all $i\in \{1,\ldots,n\}$.
Setting $q=e^{-\delta}$ and $c_b=e^{\wt b}$ for all $b\in \B$, we have for all $i \in \{0, \dots , n \}$,
\begin{align*}
\sum_{\pi\in \Pp^{\gtrdot}_{c_{b_i}}} C(\pi)q^{|\pi|} &= e^{-\Ll}\rm{ch}(L(\Ll_i)),\\
 \sum_{\pi\in \Pp^{\gg}_{c_{b_i}}} C(\pi)q^{|\pi|} &= \frac{e^{-\Ll}\rm{ch}(L(\Ll_i))}{(q;q)_{\infty}}.
\end{align*}
\end{theo}

In \cite{DK19}, we had a result of the same shape as Theorem \ref{th:charwithenergy} for $A_n^{(1)}$ which generalised Primc's identities. In addition, through a bijection, we were able to deduce a generalisation of Capparelli's identity as well. In this paper we give a generalisation of this bijection, stated in full generality in Section \ref{sec:deletingcolour}, which allows us in particular to give yet another expression for the character.

Let $\Co=\{c_b:b\in \B\}$ be the set of colours indexed by the vertices of $\B$, let $\Sc=\Co\setminus\{c_\emptyset\}$, and let $c_\infty$ be a colour not in $\Co$.
Then, for all $i\in \{0,\ldots,n\}$,
define the function $\rho_i$ on $\Sc \times (\Sc \sqcup \{c_{\infty}\})$ by
\begin{itemize}
\item $\rho_i(c_{b_i},c_\infty)=1$,
\item $\rho_i(c_b,c_\infty)=H(b_i \ot b)$ if $c_b \in \Sc \setminus \{c_{b_i}\}$,
\item $\rho_i(c_{x',y'},c_{x,y})=\chi(x\geq x')+\chi(y\geq y')-\chi(y\geq y'>x \geq x')$.
\end{itemize}
Let $\tilde{\Pp}_{\rho_i}^{c_\infty}$ be the set of generalised coloured partitions
$\pi=(\pi_0,\ldots,\pi_{s-1},\pi_s=0_{c_\infty})$ with colours in $\Sc \sqcup \{c_{\infty}\}$ such that for all $i\in \{0,\ldots,s-1\}$, 
$$c(\pi_i)\neq c_\infty \text{ and }|\pi_i|-|\pi_{i+1}|\geq \rho_i(c(\pi_i),c(\pi_{i+1})).$$
Let $\Pp_{i,\rho}$ be the set obtained from $\tilde{\Pp}_{\rho_i}^{c_\infty}$ by transforming the final part $0_{c_\infty}$ into $0_{c_{b_i}}$. We prove the following.

\begin{theo}
\label{th:bijCn}
For all $i\in \{0,\ldots,n\}$, there exists a bijection $\Phi$ between $\Pp^{\gg}_{c_{b_i}}$ 
and the product set 
 $\Pp_{i,\rho} \times \Pp$, where $\Pp$ is the set of partitions.
Furthermore, for $\Phi(\lambda)=(\mu,\nu)$, we have $|\lambda|=|\mu|+|\nu|$,  $\ell(\lambda)=\ell(\mu)+\ell(\nu)$, and the colour sequence of $\lambda$ restricted to the colours in $\Co \setminus \{c_{b_k}: k \in \{0, \dots n \}\} $ is the same as the colour sequence of $\mu$ restricted to the colours in $\Co \setminus \{c_{b_k}: k \in \{0, \dots n \}\} $.
\end{theo}

Combining Theorems \ref{th:charwithenergy} and \ref{th:bijCn}, we obtain a different character formula.
\begin{theo}
\label{th:CnCapparelli}
Setting $q=e^{-\delta}$ and $c_b=e^{\wt b}$ for all $b\in \B$, we have for all $i \in \{0, \dots , n \}$,
$$
\sum_{\pi \in \Pp_{i,\rho}} C(\pi)q^{|\pi|}= e^{-\Lambda_i}\ch(L(\Lambda_i)).
$$
\end{theo}

In \cite{DK19} we had also expressed our level $1$ characters of $A_n^{(1)}$ in terms of \textit{coloured Frobenius partitions}, i.e. two-rowed arrays of coloured integers
$$\begin{pmatrix}
\lambda_1 & \lambda_2 & \cdots & \lambda_s \\
\mu_1 & \mu_2 & \cdots & \mu_s
\end{pmatrix},$$
where $s$ is a non-negative integer.

Here, in the case of $C_n^{(1)}$, we can also make a connection with coloured Frobenius partitions, but they must satisfy additional conditions.

Let $\Rr=\{c_u: u\in \ov{[n]}\}$ be a set of colours which we call \textit{primary colours}. Recall the order \eqref{eq:orderbar} on $\ov{[n]}$ and define the corresponding order $\geq$ on $\Z_\Rr$ the set of primary-coloured integers in the following way:
\begin{equation}\label{eq:order'}
k_{c_u}\geq l_{c_v}\Longleftrightarrow k-l\geq \chi(u<v).
\end{equation}
The corresponding strict order $>$ is then defined by $k_{c_u}> l_{c_v}$ if and only if $k-l\geq \chi(u\leq v)$.
\begin{deff}
\label{def:grounded_frob}
Let $k_{c},l_{d}\in \Z_\Rr$ be two primary-coloured integers such that $l_{d}\leq k_{c}\leq (l+1)_{d}$. A \textit{$C_n^{(1)}$-Frobenius partitions with relation $>$ and ground $k_c,l_d$} is  a pair of generalised coloured partitions $(\mu,\nu)$ with parts in $\Z_\Rr$ such that 
\begin{itemize}
\item $\mu=(\mu_0,\ldots,\mu_{s-1},\mu_s=k_{c})$ and $\nu=(\nu_0,\ldots,\nu_{s-1},\nu_s=l_{d})$ are well-ordered according to $>$,
\item $\nu_j+1\geq \mu_j\geq \nu_j$ for all $j\in \{0,\ldots,s-1\}$.
\end{itemize}
The size and colour sequence of $(\mu,\nu)$ as defined as
$$|(\mu,\nu)|=\sum_{j=0}^{s-1}|\mu_j|+|\nu_j|\ \text{ and }\ C(\mu,\nu)=\prod_{j=0}^{s-1} c(\mu_j)c(\nu_j),$$
respectively (note that we do not take the fixed last parts $k_c$ and $l_d$ into account).
Let $\F^{k_{c},l_{d}}_{>}$ denote the set of $C_n^{(1)}$-Frobenius partitions with relation $>$ and ground $k_{c},l_{d}$.
\end{deff}

We give a bijection (Theorem \ref{th:bijPrhoFrob}) between $\Pp_{i,\rho}$ and $\F^{0_{c_{\ov{i+1}}},0_{c_i}}_>$ (resp. $\F^{0_{c_{\ov{1}}},-1_{c_{\ov{1}}}}_>$) for $i \in \{1, \dots , n \}$ (resp. $i=0$). Hence we get yet another expression for the character.
\begin{theo}
\label{th:CnFrob}
Setting $q=e^{-\delta}$ and $c_b=e^{\wt b}$ for all $b\in \B$, we have
$$
\sum_{(\mu,\nu) \in \F^{0_{c_{\ov{1}}},-1_{c_{\ov{1}}}}_>} C((\mu,\nu))q^{|(\mu,\nu)|}= e^{-\Lambda_0}\ch(L(\Lambda_0)),
$$
and for all $i \in \{1, \dots , n \}$,
$$
\sum_{(\mu,\nu) \in \F^{0_{c_{\ov{i+1}}},0_{c_i}}_>} C((\mu,\nu))q^{|(\mu,\nu)|}= e^{-\Lambda_i}\ch(L(\Lambda_i)).
$$
\end{theo}

Finally we make the connection between $\Pp_{i,\rho}$ and partitions satisfying frequencies restrictions on paths, in order to prove (a refinement of) the CMPP conjecture for the level $1$ standard modules.

Let $\Sc:=\{c_{x,y}=c_xc_y=c_yc_x: 1 \leq x\leq y \leq \ov{1}\}$ be the set of \textit{secondary colours} (note that it coincides with the set $\Sc$ defined above), let $\Z_\Sc$ be the corresponding set of secondary-coloured integers, and set  
$$\Z_\Sc^+=\{0_{c_{x,y}}: \ov{1}\geq y\geq x>\overline{y}\geq 1\}\sqcup(\Z_{>0})_\Sc.$$
Set $\omega_0:=(-1)_{c_{\ov{1},\ov{1}}}$, and for $i\in \{1,\ldots,n\}$,
$\omega_i:=0_{c_{i,\overline{i+1}}}$, with the convention that $\ov{n+1}=n$. Let 
$$\Omega = \{\omega_u:u\in \{0,\ldots,n\}\}\sqcup \{0_{c_{u,\overline{u}}}: u\in \{1,\ldots,n\}\}.$$
Define the function $\rho$ on $\Sc^2$ by 
$$\rho(c_{x',y'},c_{x,y}):=\chi(x\geq x')+\chi(y\geq y')-\chi(y\geq y'>x \geq x').$$ 
The corresponding relation $\gg_\rho$ on $\Z_\Sc$ is given by 
$$k_c\gg_\rho l_d \text{ if and only if }k-l\geq \rho(c,d).$$

Let $\Pp_\Sc$ denote the set of generalised coloured partitions with parts in $\Z_\Sc^+$ and order $\geq$, where order $\geq$ is defined in Section \ref{sec:frequenciesonpaths}.

Recall that any partition $\pi\in \Pp_\Sc$ can be described as its frequency sequence $(f_u)_{u\in \Z_\Sc^+}$, where $f_u$ is the number of occurrences of $u$ in $\pi$.   We then have
$$C(\pi)q^{|\pi|}=\prod_{u\in \Z_\Sc^+} \left(c(u)q^{|u|}\right)^{f_u}.$$
In Section \ref{sec:frequenciesonpaths}, we define a notion of \textit{paths} in $\Omega\sqcup\Z_\Sc^+$. We do not give the definition in this introduction as it requires too many preliminary definitions.
For $\omega\in \Omega$, let $\Pp_\Sc^\omega$ be the set of partitions of $\Pp_\Sc$ whose frequency sequence $(f_u)_{u\in \Z_\Sc^+}$ is such that, by considering fictitious occurrences of elements in $\Omega$ with $f_{\omega}=1$,  we have
$f_{e_0}+\ldots+f_{e_{m}}\leq 1$
for all paths $(e_0,\ldots,e_{m})$ in $\Omega\sqcup\Z_\Sc^+$.

We prove (in Theorem \ref{theo:mainns}, which is actually more general) that there is a bijection between $\Pp_{i,\rho}$ and $\Pp_\Sc^{\omega_i}$ which preserves the size and the colour sequence when omitting the part $\omega_i$.
From this, we deduce one last character formula.
\begin{theo}
\label{th:CnPaths}
Setting $q=e^{-\delta}$ and $c_b=e^{\wt b}$ for all $b\in \B$, we have for all $i \in \{0, \dots , n \}$,
$$
\sum_{\pi \in \Pp_\Sc^{\omega_i}} C(\pi)q^{|\pi|}= e^{-\Lambda_i}\ch(L(\Lambda_i)),
$$
where the fictitious parts are not taken into account in the generating function.
\end{theo}

Writing, for $k\in \Z\cup \{-\infty\}$,
$$E_k:=\{l_d: l\geq k,\,d\in\{0,\ldots,2n\},\, l-d\equiv 1 \mod 2\},$$  
a \textit{path in $E_k$} is a sequence $((l^{(0)})_0, \dots , (l^{(m)})_{2n})$ such that for all $j \in \{0, \dots , 2n-1\}$, $l^{(j+1)}= l^{(j)} \pm 1$  and $(l^{(j)})_j \in E_k$ ($k \in \Z \cup \{- \infty \}$).

\noindent Hence the set of paths $(e_0,\ldots,e_{2n})$ in $E_{-1}$ is a formal expression for the set of admissible paths described graphically in the CMPP conjecture.

Thus the CMPP conjecture can be reformulated as follows.
\begin{con}[Reformulation of the CMPP conjecture]
\label{con:CMPPreformulated}
Let $k_0,\ldots,k_n$ be non-negative integers and denote by $\Pp^{k_0,\ldots,k_n}_{\mathds{1}}$ the set of partitions with parts in $E_1$ such that, letting $f_u$ be the frequency of $u$ for all $u\in E_1$, and setting fictitious frequencies $f_{(-1)_{2n-2i}}=k_i$ for $i\in \{0,\ldots,n\}$, we have
$$f_{e_0}+\cdots+f_{e_{2n}}\leq k_0+\cdots+k_n$$
for all paths $(e_0,\ldots,e_{2n})$ in $E_{-1}$. Then,
$$\sum_{\pi\in \Pp^{k_0,\ldots,k_n}_{\mathds{1}}}q^{|\pi|}= \mathds{1}\left(e^{-k_0\Lambda_0-\cdots-k_n\Lambda_n}\ch(L(k_0\Lambda_0+\cdots+k_n\Lambda_n))\right)\,.$$
\end{con}

On the other hand, the principal specialisation of Theorem \ref{th:CnPaths} gives the following.
\begin{theo}[CMPP for $L(\Lambda_i)$, principal specialisation of Theorem \ref{th:CnPaths}] \label{theo:mainsp}
Let $i\in \{0,\ldots,n\}$, and denote by $\Pp^{(-1)_{2n-2i}}$ the set of partitions with parts in $E_1$ such that, letting $f_u$ be the frequency of $u$ for all $u\in E_1$, and setting fictitious frequencies $f_u=\chi(u=(-1)_{2n-2i})$ for all $u\in E_{-1}\setminus E_1$, we have
$$f_{e_0}+\cdots+f_{e_{2n}}\leq 1$$
for all paths $(e_0,\ldots,e_{2n})$ in $E_{-1}$. Then,
$$\sum_{\pi\in \Pp^{(-1)_{2n-2i}}}q^{|\pi|}= \mathds{1}(e^{-\Lambda_i}\ch(L(\Lambda_i))) = \frac{(q^{2n+4},q^{2i+2},q^{2n-2i+2};q^{2n+4})_\infty}{(q;q)_\infty}.$$
\end{theo}

\noindent Here the product formula follows from \eqref{eq:princchar}.

Theorem \ref{th:CnPaths}, which is a \textit{non-specialised} version of Theorem \ref{theo:mainsp}, gives us a good hint for what the \textit{non-specialised generalisation} of the CMPP conjecture could be. We conjecture the following.
\begin{con}[Non-specialised CMPP conjecture]
Let $k_0,\ldots,k_n$ be non-negative integers, and denote by $\Pp^{k_0,\ldots,k_n}_\Sc$ the set of partitions with parts in $\Z_\Sc^+$ such that, if $(f_u)_{u\in \Z_\Sc^+}$ are the frequencies of the parts and if we consider fictitious frequencies $f_{\omega_i}=k_i$ for all $i\in\{0,\ldots,n\}$, we have 
$$f_{e_0}+\ldots+f_{e_{2n}}\leq k_0+\ldots+k_n$$
for all paths $(e_0,\ldots,e_{2n})$ of $\Omega\sqcup\Z_\Sc^+$. By setting $q=e^{-\delta}$, $c_i=e^{\frac{\alpha_n}{2}+\sum_{u=i}^{n-1}\alpha_u}$ and $c_{\overline{i}}=c_i^{-1}$ for all $i\in\{1,\ldots,n\}$, we have 
\begin{equation*}
\sum_{\pi\in \Pp^{k_0,\ldots,k_n}_\Sc} C(\pi)q^{|\pi|} = e^{-k_0\Lambda_0-\cdots-k_n\Lambda_n}\ch(L(k_0\Lambda_0+\cdots+k_n\Lambda_n))\,.
\end{equation*}
\end{con}
Studying perfect crystals of $C_n^{(1)}$ of higher levels may be a good way to approach this conjecture.

\medskip
The paper is organised as follows. In Section \ref{sec:energy}, we describe the level $1$ perfect crystal of $C_n^{(1)}$ more rigorously and we prove the formula for the energy function given in Theorem \ref{th:energy}. 
In Section \ref{sec:deletingcolour}, we give the general version of the bijection between $\Pp^{\gg}_{c_{b_i}}$ and $\Pp_{i,\rho} \times \Pp$, and prove Theorem \ref{th:bijCn} and \ref{th:CnCapparelli}.
In Section \ref{sec:roadfrob}, we describe the bijection between $\Pp_{i,\rho}$ and  $C_n^{(1)}$-Frobenius partitions in a general setting.
In Section \ref{sec:frequenciesonpaths}, we describe the bijection with partitions described by frequencies on paths in a general setting and consider some particular cases of our general bijections to establish the CMPP conjecture for all level $1$ standard modules of $C_n^{(1)}$.

\section{A level $1$ perfect crystal of $C_n^{(1)}$}
\label{sec:energy}
\subsection{Description of the crystal}
The following perfect crystal of type $C_n^{(1)}$ was introduced by Kang, Kashiwara and Misra in \cite{KKM} together with an energy function. Here we use the particular case of level $1$ as stated in \cite{HKL}.

\begin{theo}[Kang--Kashiwara--Misra 1994]
A perfect crystal of level $1$ of type $C_n^{(1)}$ can be defined as 
\begin{equation}
\label{eq:crystalKKM}
\B=
 \Big\{ (x_1,\dots,x_n, \ov{x}_n,\dots, \ov{x}_1) \,\Big\vert\,
x_i, \ov{x}_i \in \Z_{\geq0},\
\textstyle\sum_{i=1}^{n} (x_i + \ov{x}_i) = \textnormal{$0$ or $2$} \Big\}.
\end{equation}
The action of the Kashiwara operator $\tilde{f}_i$ on $b = (x_1,\dots,x_n, \ov{x}_n,\dots, \ov{x}_1) \in \B$ is given by
$$\tilde{f}_0 b =
\begin{cases}
(x_1 + 2, x_2, \dots, \ov{x}_2, \ov{x}_1) &
 \text{ if } x_1 \geq \ov{x}_1,\\
(x_1 + 1, x_2, \dots, \ov{x}_2, \ov{x}_1 - 1) &
 \text{ if } x_1 = \ov{x}_1 - 1,\\
(x_1, x_2, \dots, \ov{x}_2, \ov{x}_1 - 2) &
 \text{ if } x_1 \leq \ov{x}_1 - 2,
\end{cases}$$
for $i=1,\dots,n-1$,
$$\tilde{f}_i b =
\begin{cases}
(x_1, \dots, x_i - 1, x_{i+1} + 1, \dots, \ov{x}_1) &
\text{ if } x_{i+1} \geq \ov{x}_{i+1},\\
(x_1, \dots, \ov{x}_{i+1}-1, \ov{x}_{i} + 1, \dots, \ov{x}_1) &
\text{ if } x_{i+1} < \ov{x}_{i+1},
\end{cases}$$
and
$$\tilde{f}_n b = (x_1, \dots, x_n - 1, \ov{x}_n + 1, \dots, \ov{x}_1).$$
Conversely, the action of the operator $\eit$ on $\B$ is given by
$$\tilde{e}_0 b =
\begin{cases}
(x_1 - 2, x_2, \dots, \ov{x}_2, \ov{x}_1) &
 \text{ if } x_1 \geq \ov{x}_1 + 2,\\
(x_1 - 1, x_2, \dots, \ov{x}_2, \ov{x}_1 + 1) &
 \text{ if } x_1 = \ov{x}_1 + 1,\\
(x_1, x_2, \dots, \ov{x}_2, \ov{x}_1 + 2) &
 \text{ if } x_1 \leq \ov{x}_1,
\end{cases}$$
for $i=1,\dots,n-1$,
$$\eit b =
\begin{cases}
(x_1, \dots, x_i + 1, x_{i+1} - 1, \dots, \ov{x}_1) &
\text{ if } x_{i+1} > \ov{x}_{i+1},\\
(x_1, \dots, \ov{x}_{i+1} + 1, \ov{x}_{i} - 1, \dots, \ov{x}_1) &
\text{ if } x_{i+1} \leq \ov{x}_{i+1},
\end{cases}$$
and
\begin{equation*}
\tilde{e}_n b = (x_1, \dots, x_n + 1, \ov{x}_n - 1, \dots, \ov{x}_1).
\end{equation*}

Moreover, an energy function for $\B$ is given, for all $b = (x_1,\dots,x_n, \ov{x}_n,\dots, \ov{x}_1)$ and

\noindent  $b' = (x'_1,\dots,x'_n, \ov{x}'_n,\dots, \ov{x}'_1) \in \B$, by

\begin{equation}
\label{eq:KKMenergy}
H(b \ot b') = \max \big\{ \theta_j(b \ot b') ,  \theta'_j(b \ot b'),  \eta_j(b \ot b'),  \eta'_j(b \ot b') : j \in \{1, \dots , n \} \big\},
\end{equation}
where
\begin{align*}
\theta_j(b \ot b') &= \sum_{k=1}^{j-1} (\ov{x}_k - \ov{x}_k') + \frac{s(b')-s(b)}{2},\\
\theta'_j(b \ot b') &= \sum_{k=1}^{j-1} (x'_k - x_k) + \frac{s(b)-s(b')}{2},\\
\eta_j(b \ot b') &= \sum_{k=1}^{j-1} (\ov{x}_k - \ov{x}_k') + (\ov{x}_j - x_j) + \frac{s(b')-s(b)}{2},\\
\eta'_j(b \ot b') &= \sum_{k=1}^{j-1} (x_k' - x_k) + (x'_j - \ov{x}'_j) + \frac{s(b)-s(b')}{2},
\end{align*}
and
$$s(b) = \sum_{k=1}^n x_k + \sum_{k=1}^n \ov{x}_k.$$
\end{theo}
Using the set
$$\overline{[n]}= \{1, \dots , n, \overline{n} , \dots , \overline{1}\}$$
and order
$$1<2<\cdots <n-1<n < \overline{n}< \overline{n-1}<\cdots<\overline{2}<\overline{1}$$
defined in the introduction,
the crystal $\B$ can be rewritten as
$$\B = \{(0, \dots, 0) \} \sqcup \Big\{ e_x + e_y \,\Big\vert\, x \leq y,\ x, y \in \overline{[n]} \Big\},$$
where $e_j$ is the vector with entries indexed by $\overline{[n]}$ having all entries $0$ except a $1$ in position $j$.

Thus we can identify $\B$ with the crystal
$$\{\emptyset \} \sqcup \{(x,y) | x\leq y \in \overline{[n]}\},$$
via the correspondence
\begin{align*}
(0, \dots, 0) &\leftrightarrow \emptyset,\\
e_x + e_y  &\leftrightarrow (x,y) \text{ for all }  1 \leq x \leq y \leq \ov{1}.
\end{align*}
From now on, for simplicity of notation, we will always consider that
$$\B = \{\emptyset \} \sqcup \{(x,y) | x\leq y \in \overline{[n]}\},$$
as shown on Figure \ref{fig:crystalb}.

Moreover, with the same correspondence, the functions $\theta_j, \theta_j', \eta_j , \eta_j'$ can be rewritten as follows.

\begin{prop}
For all  $x\leq y \in \overline{[n]}$ and $j \in \{1, \dots, n\}$, we have
\begin{align*}
\theta_j((x,y) \ot (x',y')) &= \chi(x > \ov{j}) + \chi(y > \ov{j}) - \chi(x'> \ov{j}) - \chi(y' >\ov{j}),\\
\theta'_j((x,y) \ot (x',y')) &= \chi(x' < j) + \chi(y' < j) - \chi(x < j) - \chi(y< j),\\
\eta_j((x,y) \ot (x',y')) &= \chi(x \geq \ov{j}) + \chi(y \geq \ov{j}) - \chi(x'> \ov{j}) - \chi(y' >\ov{j}) - \chi(x=j) - \chi(y=j),\\
\eta'_j((x,y) \ot (x',y')) &= \chi(x' \leq j) + \chi(y' \leq j) - \chi(x < j) - \chi(y< j) - \chi(x'= \ov{j}) - \chi(y'= \ov{j}).
\end{align*}
\end{prop}
The proof follows immediately from the correspondence above and the fact that $x_k=\chi(x = k) + \chi(y = k) $.

\subsection{Rewriting the energy function}
Now that the perfect crystal and energy function have been defined, we prove our alternative formulation of the energy function from Theorem \ref{th:energy}. 
To do so, we show that the energy function $H$ defined in \eqref{eq:KKMenergy} is actually the same as the function $H'$ defined by
$$
\begin{cases}
H'(\emptyset\ot \emptyset)=0,\\
H'(\emptyset\ot (x,y))=H((x,y)\ot\emptyset )=1,\\
\end{cases}
$$
and
\begin{equation}
\label{eq:Hdef}
H'((x,y)\ot (x',y')) = \begin{cases}
\chi(x\geq x')+\chi(y\geq y')-\chi(y\geq y'>x \geq x') \ \ \text{if} \ \ \overline{y'}\neq x, \\
\chi(x>x')+\chi(y> y')-\chi(y> y'>x > x') \ \ \text{if} \ \ \overline{y'}= x.
\end{cases}
\end{equation}
We prove this in several steps. First, we prove the following lemma.

\begin{lem}
For all $x \leq y \in \ov{[n]}$,
$$
\begin{cases}
H(\emptyset\ot \emptyset)=0,\\
H(\emptyset\ot (x,y))=H((x,y)\ot\emptyset )=1.
\end{cases}
$$
\end{lem}
\begin{proof}
By definition, for all $1 \leq j \leq n$,
$$\theta_j(\emptyset \ot \emptyset) = \theta'_j(\emptyset \ot \emptyset) = \eta_j(\emptyset \ot \emptyset) = \eta'_j(\emptyset \ot \emptyset) = 0,$$
and hence $H(\emptyset\ot \emptyset)=0$.

Now let $x \leq y \in \ov{[n]}$. We have for all $1 \leq j \leq n$,
\begin{align*}
\theta_j((x,y) \ot \emptyset) &= \chi(x > \ov{j}) + \chi(y > \ov{j}) - 1,\\
\theta'_j((x,y) \ot \emptyset) &= - \chi(x < j) - \chi(y< j) +1,\\
\eta_j((x,y) \ot \emptyset) &= \chi(x \geq \ov{j}) + \chi(y \geq \ov{j}) - \chi(x=j) - \chi(y=j)-1,\\
\eta'_j((x,y) \ot \emptyset) &=  - \chi(x < j) - \chi(y< j) +1.
\end{align*}
Thus for all $1 \leq j \leq n$, the above quantities are at most $1$. Moreover,  for all $x \leq y \in \ov{[n]}$, we have $\theta'_1((x,y) \ot \emptyset)=1$, and thus $H((x,y) \ot \emptyset)=1$.

The proof of $H(\emptyset \ot (x,y))=1$ works in the same way.
\end{proof}

Now let us show that $H((x,y)\ot (x',y'))$ and $H'((x,y)\ot (x',y'))$ coincide for all $x \leq y, x' \leq y' \in \ov{[n]}$.
First we need a lemma about $H'$.

\begin{lem}
\label{lem:H012}
For all $x \leq y, x' \leq y' \in \ov{[n]}$, $0 \leq H'((x,y)\ot (x',y')) \leq 2$.
Moreover,
\begin{equation}\label{eq:h=2}
H'((x,y)\ot (x',y')) = 2\text{ if and only if }x\geq y',
\end{equation}
and 
\begin{align}
H'((x,y)\ot (x',y')) = 0 \text{ if and only if either }&x< x'\text{ and }y< y'\label{eq:asrho}\\
\text{or}\quad&\overline{y}\geq x = \overline{y'}=x'\label{eq:inf}\\
\text{or}\quad&\overline{y}= x = \overline{y'}< x'.\label{eq:sup}
\end{align}
In the other cases, $H'((x,y)\ot (x',y'))=1$.
\end{lem}

\begin{rem}
\label{rem:rk1}
When $x,x',y,y'$ are not in Cases \eqref{eq:inf} and \eqref{eq:sup}, $H'((x,y)\ot (x',y'))$ is always equal to the first line of \eqref{eq:Hdef}.
\end{rem}

\begin{proof}
The function $\chi$ has values in $\{0,1\}$, so from \eqref{eq:Hdef}, $H'$ has values in $\{-1,0,1,2\}$. But in both lines of \eqref{eq:Hdef}, the argument of the last $\chi$ contains the arguments of the first two $\chi$'s, so the last $\chi$ can only be $1$ whenever the first two are $1$ as well. Therefore $H'$ never evaluates to $-1$. The first statement is proved.

\medskip
Now let us prove \eqref{eq:h=2}. By the previous argument, if $\overline{y'}\neq x$, we have $H'((x,y)\ot (x',y')) = 2$ iff $\chi(x\geq x')=\chi(y\geq y')=1$ and $\chi(y\geq y'>x \geq x')=0$, i.e. iff $x\geq x'$, $y\geq y'$ and $x \geq y'$. But we always have $x \leq y$ and $x' \leq y'$, so these three conditions are equivalent to $x \geq y'$ alone.

Similarly, if $\overline{y'}= x$, we have $H'((x,y)\ot (x',y')) = 2$ iff $\chi(x> x')=\chi(y> y')=1$ and $\chi(y> y'>x > x')=0$, i.e. iff $x > x'$, $y > y'$ and $x \geq y'$. As before, the conditions $x \geq x'$ and $y \geq y'$ are already included in $x \geq y'$. Moreover, we never have ($x=x'$ and $x \geq y'$) or ($y=y'$ and $x \geq y'$) at the same time, because it would mean that $y'=x'=x=\overline{y'}$, which is impossible. Thus if $\overline{y'}= x$, $x > x'$, $y > y'$ and $x \geq y'$ iff $x \geq y'$.

Summarising, we have $H'((x,y)\ot (x',y')) = 2$ iff ($\overline{y'}\neq x$ and $x \geq y'$) or ($\overline{y'}= x$ and $x \geq y'$), i.e. iff $x \geq y'$.

\medskip
Finally we treat the cases where $H'((x,y)\ot (x',y')) =0$. Observing that in the definition of $H'$, the last $\chi$ can only be $1$ whenever the two first are $1$ too, we deduce that $H'$ is equal to $0$ if and only if the three $\chi$'s in its definition are all equal to $0$, i.e. iff ($\overline{y'}\neq x$ and $x < x'$ and $y < y'$) or ($\overline{y'}= x$ and $x \leq x'$ and $y \leq y'$). The statement ($x \leq x'$ and $y \leq y'$) can be decomposed as (($x < x'$ and $y < y'$) or ($x=x'$ and $y \leq y'$) or ($x<x'$ and $y=y'$)). Thus $H'((x,y)\ot (x',y'))=0$ if and only if we are in one of the three cases \eqref{eq:asrho}--\eqref{eq:sup}.
\end{proof}

From the definition of $H$ and the fact that for all $x \leq y, x' \leq y' \in \ov{[n]}$, $\theta_1((x,y) \ot (x',y'))=0$, we know that $H$ takes values in $\{0,1,2\}$.
So we only have to prove that the values of $(x,y)\ot (x',y')$ for which $H$ and $H'$ are equal to $2$ (resp. $1$, $0$) coincide.

\begin{lem}
\label{lem:H2}
For all $x \leq y, x' \leq y' \in \ov{[n]}$,
$$H((x,y)\ot (x',y')) = 2\text{ if and only if }x\geq y'.$$
\end{lem}
\begin{proof}
We first prove that if $x\geq y'$, then $H((x,y)\ot (x',y')) = 2$. We have the following:
\begin{itemize}
\item If $x> y'$ and $x > \ov{n}$, then for $\ov{j}=x-1$, we have $\theta_j((x,y)\ot (x',y'))=2$.
\item If $x> y'$ and $x \leq n$, then for $j=y'+1$, we have $\theta'_j((x,y)\ot (x',y'))=2$.
\item If $x> y'$ and $x = \ov{n}$, then $\eta_n((x,y)\ot (x',y'))=2$.
\item If $x=y'$ and $x \geq \ov{n}$, then for $\ov{j}=x$, we have $\eta_j((x,y)\ot (x',y'))=2$.
\item If $x=y'$ and $x \leq n$, then for $j=x$, we have $\eta'_j((x,y)\ot (x',y'))=2$.
\end{itemize}
In all the cases above, i.e. for $x\geq y'$, we have indeed that $H((x,y)\ot (x',y')) = 2$.

\medskip
Now we prove that if $x<y'$, then $H((x,y)\ot (x',y')) \leq 1$. If $x<y'$, we have:
\begin{itemize}
\item For all $j$, $\theta_j((x,y)\ot (x',y'))\leq 1$, because if $\chi(x>\ov{j})=1$, then also $\chi(y'>\ov{j})=1$.
\item For all $j$, $\theta'_j((x,y)\ot (x',y'))\leq 1$, because if $\chi(y'<j)=1$, then also $\chi(x<j)=1$.
\item For all $j$, $\eta_j((x,y)\ot (x',y'))\leq 1$, because if $\chi(x\geq\ov{j})=1$, then also $\chi(y'>\ov{j})=1$.
\item For all $j$, $\eta'_j((x,y)\ot (x',y'))\leq 1$, because if $\chi(y' \leq j)=1$, then also $\chi(x<j)=1$.
\end{itemize}
Thus as soon as $x<y'$, we have $H((x,y)\ot (x',y')) \leq 1$.
\end{proof}

Now let us turn to the values for which $H=0$.
\begin{lem}
\label{lem:H0}
Let $x \leq y, x' \leq y' \in \ov{[n]}$. If $x,y,x',y'$ are in Cases \eqref{eq:asrho}, \eqref{eq:inf} or \eqref{eq:sup}, then $H((x,y)\ot (x',y')) =0$.
\end{lem}
\begin{proof}
Note that, by definition of $\theta_j$ and $\theta'_j$, if $x \leq x'$ and $y \leq y'$, then $\theta_j \leq 0$ and $\theta'_j \leq 0$ for all $j$. Similarly, if $x < x'$ and $y < y'$, then $\eta_j \leq 0$ and $\eta'_j \leq 0$ for all $j$.
So we already know that when we are in Case \eqref{eq:asrho}, i.e. $x < x'$ and $y < y'$, we have $H((x,y)\ot (x',y')) =0$.

\medskip
Now let us study Case \eqref{eq:inf}, i.e. $\overline{y}\geq x = \overline{y'}=x'$.
When $x=x'$, the only $j$ for which $\eta_j$ could be equal to $1$ is $j=\ov{x}$, which can only happen if $x \in \{\ov{n}, \dots , \ov{1}\}$. Indeed it is the only case where $\chi(x \geq \ov{j})- \chi(x'> \ov{j})=1$. But then we do not have $x \leq \ov{y}$ (because  $x \leq y$ and $x$ is an overlined number), so we never have $\eta_j ((x,y)\ot (x',y'))= 1$ in Case \eqref{eq:inf}. 

Similarly, the only $j$ for which $\eta'_j$ could be equal to $1$ is $j=x$, which can only happen if $x \in \{1, \dots , n\}$. In that case we have
$$\eta'_j ((x,y)\ot (x',y'))= 1 - \chi(y'=\ov{j}),$$
which equals $0$ iff $y'=\ov{j}=\ov{x}$.

Thus in Case \eqref{eq:inf}, for all $j$, the functions $\theta_j, \theta'_j , \eta_j $ and $\eta'_j$ are at most $0$ and thus $H((x,y)\ot (x',y')) =0$.

\medskip
Finally we treat Case \eqref{eq:sup}, i.e. $\overline{y}= x = \overline{y'}< x'$.
When $y=y'$, the only $j$ for which $\eta_j$ could be equal to $1$ is $j=\ov{y}$, which can only happen if $y \in \{\ov{n}, \dots , \ov{1}\}$.
In that case we have
$$\eta_j ((x,y)\ot (x',y'))= 1 - \chi(x=j),$$
which equals $0$ iff $x=j=\ov{y}$.

Similarly, the only $j$ for which $\eta'_j$ could be equal to $1$ is $j=y$, which can only happen if $y \in \{1, \dots , n\}$.
But then we do not have $x=\ov{y}$ (because  $x \leq y$ and $y$ is an non-overlined number), so we never have $\eta_j ((x,y)\ot (x',y'))= 1$ in Case \eqref{eq:sup}. 

Thus in Case \eqref{eq:sup}, for all $j$, the functions $\theta_j, \theta'_j , \eta_j $ and $\eta'_j$ are at most $0$ and $H((x,y)\ot (x',y')) =0$.
\end{proof}

The only thing left to do is prove that in all the cases not treated in Lemmas \ref{lem:H2} and \ref{lem:H0}, $H$ has value $1$.

\begin{lem}
\label{lem:H1}
In all the other cases, i.e. if
\begin{align}
&x < y' \text{ and } (x >x' \text{ or } y >y'), \label{eq:case1}\\
\text{or } &x < y' \text{ and } x =x' \neq \ov{y'} \text{ and } {y \leq y'} , \label{eq:case2a}\\
\text{or } &x < y' \text{ and } y =y' \neq \ov{x} \text{ and } {x \leq x'}, \label{eq:case3a}
\end{align}
we have
 $$H((x,y)\ot (x',y')) =1.$$
\end{lem}
\begin{proof}
Thanks to Lemmas \ref{lem:H012} and \ref{lem:H2}, we know that $H$ takes values in $\{0,1\}$ when $x < y'$. Thus to prove this lemma, we only need, for all the possible values of $x,y,x',y'$ under consideration, to exhibit some $\theta_j, \theta'_j, \eta_j$ or $\eta'_j$ which are equal to $1$.

\medskip
Let us start with Case \eqref{eq:case1}.
\begin{itemize}
\item If $x>x'$, $x < y'$ and $x=\ov{j}$ for some $j \in \{1, \dots , n-1\}$, then $\theta_{j +1}((x,y)\ot (x',y'))= 1$.
\item If $x>x'$, $x < y'$ and $x=\ov{n}$, then $\eta_{n}((x,y)\ot (x',y'))= 1$.
\item If $x>x'$, $x < y'$ and $x=j$ for some $j \in \{1, \dots , n\}$, then $\theta'_{j}((x,y)\ot (x',y'))= 1$.
\item If $y>y'$, $x < y'$ and $y=\ov{j}$ for some $j \in \{1, \dots , n-1\}$, then $\theta_{j +1}((x,y)\ot (x',y'))= 1$.
\item If $y>y'$, $x < y'$ and $y=\ov{n}$, then $\eta_{n}((x,y)\ot (x',y'))= 1$.
\item If $y>y'$, $x < y'$ and $y=j$ for some $j \in \{1, \dots , n\}$, then $\theta'_{j}((x,y)\ot (x',y'))= 1$.
\end{itemize}
So we always have $H((x,y)\ot (x',y')) =1$ in Case \eqref{eq:case1}.

\medskip
Now we consider Case \eqref{eq:case2a}.
\begin{itemize}
\item If $ x =x' \neq \ov{y'}$, $x < y'$, $y \leq y'$ and $x=\ov{j}$ for some $j \in \{1, \dots , n\}$, then $\eta_{j}((x,y)\ot (x',y'))= 1$.
\item If $ x =x' \neq \ov{y'}$, $x < y'$, $y \leq y'$ and $x=j$ for some $j \in \{1, \dots , n\}$, then $\eta'_{j}((x,y)\ot (x',y'))= 1$.
\end{itemize}
Thus we always have $H((x,y)\ot (x',y')) =1$ in Case \eqref{eq:case2a}.

\medskip
We proceed similarly for Case \eqref{eq:case3a}.
\begin{itemize}
\item If $ y =y' \neq \ov{x}$, $x < y'$, $x \leq x'$ and $y=\ov{j}$ for some $j \in \{1, \dots , n\}$, then $\eta_{j}((x,y)\ot (x',y'))= 1$.
\item If $ y =y' \neq \ov{x}$, $x < y'$, $x \leq x'$ and $y=j$ for some $j \in \{1, \dots , n\}$, then $\eta'_{j}((x,y)\ot (x',y'))= 1$.
\end{itemize}
Thus we always have $H((x,y)\ot (x',y')) =1$ in Case \eqref{eq:case3a}, and the lemma is proved.
\end{proof}

We have shown that for all $x \leq y, x' \leq y' \in \ov{[n]}$,
$H((x,y)\ot (x',y'))= H'((x,y)\ot (x',y'))$, and thus Theorem \ref{th:energy} is proved.

\section{Deleting a colour}
\label{sec:deletingcolour}
In this section we present a method given in \cite{KonanPhD} and use it to simplify the difference conditions given by the energy function $H$ introduced in Theorem \ref{th:energy}. This method is a generalisation of the connection between generalisations of Primc's and Capparelli's identities established in \cite{DK19}, which itself generalised a bijection of the first author \cite{DousseCapaPrimc} between Primc's and Capparelli's identities.

\subsection{The setup}
\begin{deff}
\label{def:welldef}
Let $\Co$ be a set of colours, and let $\Co_{\mathrm{sup}}\sqcup \Co_{\mathrm{free}} \sqcup \Co_{\mathrm{inf}}$ be a set-partition of $\Co$. Let $c_{\infty}$ be an additional colour. A function $\ep$ from $\Co \times (\Co \sqcup \{c_{\infty}\})$ to $\{0,1,2\}$ is said to be \textit{well-defined according to the decomposition $\Co_{\mathrm{sup}}\sqcup \Co_{\mathrm{free}} \sqcup \Co_{\mathrm{inf}}$ and $c_{\infty}$} if it satisfies the following conditions:
\begin{enumerate}
 \item For $c,c'\in \Co_{\mathrm{free}}$, $\ep(c,c')=\chi(c\neq c')$.
 \item For $(c,c')\in \Co_{\mathrm{sup}}\times \Co_{\mathrm{free}}$,
 $\ep(c,c')\in \{0,1\}$ and $\ep(c',c)\in \{1,2\}$.
Moreover, for $c\in \Co_{\mathrm{sup}}$, there exists $\delta(c)\in \Co_{\mathrm{free}}$ such that $\ep(c,\delta(c))=0$.
   \item For $(c,c')\in \Co_{\mathrm{free}}\times \Co_{\mathrm{inf}}$, 
$\ep(c,c')\in \{0,1\}$ and $ \ep(c',c)\in \{1,2\}$.
Moreover, for $c\in \Co_{\mathrm{inf}}$, there exists $\delta(c)\in \Co_{\mathrm{free}}$ such that $\ep(\delta(c),c)=0$.
 \item For $(c,c')\in \Co_{\mathrm{sup}}\times \Co_{\mathrm{inf}}$, 
 $\ep(c,c')\in \{0,1\}$ and $\ep(c',c)\in \{1,2\}$.
Moreover, if $\ep(c,c')=0$, there then exists $\gamma(c,c')\in \Co_{\mathrm{free}}$ such that 
$\ep(c,\gamma(c,c'))= \ep(\gamma(c,c'),c')=0$.
   \item For $c,c'\in \Co_{\mathrm{sup}}$, if $ \ep(c,c')\in \{0,1\}$, there then exists $\gamma(c,c')\in \Co_{\mathrm{free}}$ such that
$\ep(c,\gamma(c,c'))= 0$ and $\ep(\gamma(c,c'),c')=1$.
    \item For $c,c'\in \Co_{\mathrm{inf}}$, if $ \ep(c,c')\in \{0,1\}$, there then exists $\gamma(c,c')\in \Co_{\mathrm{free}}$ such that
 $\ep(c,\gamma(c,c'))= 1$ and $\ep(\gamma(c,c'),c')=0$.
 	\item Border conditions: $
 \ep(\Co_{\mathrm{free}},c_\infty) = \{1\}\,,\,
 \ep(\Co_{\mathrm{inf}},c_\infty) \subset \{1,2\}\text{ and }
  \ep(\Co_{\mathrm{sup}},c_\infty) \subset  \{0,1\}.
$
\end{enumerate}

\end{deff}

For the entirety of Sections 2.1 to 2.3, let $\Co=\Co_{\mathrm{sup}}\sqcup \Co_{\mathrm{free}} \sqcup \Co_{\mathrm{inf}}$ be a set of colours, let $c_0\in \Co_{\mathrm{free}}$, let $c_{\infty}$ be an additional colour not in $\mathcal{C}$, and let $\ep$ be a well-defined function according to the decomposition $\Co_{\mathrm{sup}}\sqcup \Co_{\mathrm{free}} \sqcup \Co_{\mathrm{inf}}$ and $c_{\infty}$. Additionally, for all $c\in \Co_{\mathrm{sup}}\sqcup \Co_{\mathrm{inf}} $, let $\delta(c)$ be a particular colour satisfying Conditions (2) and (3) of Definition \ref{def:welldef},
and for all $(c,c')$ in $\{(c,c')\in \Co_{\mathrm{sup}}\times \Co_{\mathrm{inf}}: \ep(c,c')=0\}\sqcup \{(c,c')\in \Co_{\mathrm{sup}}^2: \ep(c,c')\in \{0,1\}\}\sqcup \{(c,c')\in \Co_{\mathrm{inf}}^2: \ep(c,c')\in \{0,1\}\}$, let $\gamma(c,c')$ be a particular colour satisfying Conditions (4) to (6) of Definition \ref{def:welldef}.

Throughout the paper, if $k_c$ and $k'_{c'}$ are two coloured integers, then we use the notation 
$$k_c \gg_\ep k'_{c'} \text{ if and only if } k - k' \geq \ep(c,c').$$ 
Finally, adding an integer $l$ to a coloured integer $k_c$ consists in adding this integer to its size while keeping its colour, i.e. $k_c+l=(k+l)_c$.

\begin{deff} 
\label{def:Psimple}
Define $\Pp_{\ep}^{c_\infty}$ to be the set of generalised coloured partitions $\pi=(\pi_0,\ldots,\pi_{s-1},\pi_s=0_{c_\infty})$ such that for all $i\in \{0,\ldots,s-1\}$, $c(\pi_i)\in \mathcal{C}$ and $\pi_i \gg_\ep \pi_{i+1}$.
\end{deff}

In this paper, a \textit{pattern} denotes a finite sequence of coloured integers well-ordered according to $\gg_\ep$. 
We say that a generalised coloured partition $\pi=(\pi_0,\dots,\pi_s)$ \textit{contains the pattern} $\lambda_0, \dots, \lambda_k$ if there exists an index $i$ such that
$$\pi_i = \lambda_0,\ \pi_{i+1}=\lambda_1,\ \dots ,\ \pi_{i+k}=\lambda_k.$$ If a partition does not contain a pattern, we say that it \textit{avoids} it.

\begin{deff}
\label{def:Pcomplique}
Denote by $_{\delta,\gamma}^{c_0}\Pp_{\ep}^{c_\infty}$ the set of generalised coloured
 partitions of $\Pp_{\ep}^{c_\infty}$ which are $c_0$-regular (i.e. have no part of colour $c_0$) and avoid the following patterns for all $p \in \N_{\geq 0}$:
 \begin{enumerate}
  \item for  $c\in \Co_{\mathrm{free}}\setminus \{c_0\}$, the pattern
  $p_c,p_c$,
  \item for $(c,c')\in \Co_{\mathrm{sup}}\times \Co_{\mathrm{inf}}$ such that $\ep(c,c')=0$, the pattern $p_c,p_{\gamma(c,c')},p_{c'}$,
  \item for $(c,c')\in \Co_{\mathrm{sup}}^2$ such that $\ep(c,c')\in \{0,1\}$, the pattern $p_c,p_{\gamma(c,c')},(p-1)_{c'}$,
    \item for $(c,c')\in \Co_{\mathrm{inf}}^2$ such that $\ep(c,c')\in \{0,1\}$, the pattern $(p+1)_c,p_{\gamma(c,c')},p_{c'}$,
  \item for all $c \in \Co_{\mathrm{sup}}$,
  \begin{enumerate}
  \item for all $c' \in (\Co_{\mathrm{free}}\setminus \{c_0\} )\sqcup \Co_{\mathrm{inf}}\sqcup\{c_\infty\}$, the pattern $p_c,p_{\delta(c)},(p-1)_{c'}$,
  \item for all $c' \in (\Co\setminus\{c_0\})\sqcup\{c_\infty\}$, and for all positive integers $u\geq 2$, the pattern $p_c,p_{\delta(c)},(p-u)_{c'}$,
  \end{enumerate}
  \item for all $c' \in \Co_{\mathrm{inf}}$,
  \begin{enumerate}
   \item at the beginning of the partition, the pattern $p_{\delta(c')},p_{c'}$,
  \item for all $c \in (\Co_{\mathrm{free}}\setminus \{c_0\} )\sqcup \Co_{\mathrm{sup}}$, the pattern $(p+1)_c,p_{\delta(c')},p_{c'}$,
  \item for all $c \in \Co\setminus\{c_0\}$, and for all positive integers $u\geq 2$, the pattern $(p+u)_c,p_{\delta(c')},p_{c'}$.
  \end{enumerate}
\end{enumerate}
\end{deff}

We are now ready to state the main result of this section.
\begin{theo}\label{theo:dualitycapaprimc}
 Assume that $c_0 \in \Co_{\mathrm{free}}$ is such that for all $c\in \Co \setminus \{c_0\}$, $\ep(c_0,c)=\ep(c,c_0)=1$.
 Then, there exists a bijection $\Phi$ between $\Pp_{\ep}^{c_\infty}$ and the product set 
 $_{\delta,\gamma}^{c_0}\Pp_{\ep}^{c_\infty} \times \Pp$, where $\Pp$ is the set of classical integer partitions.
Furthermore, for $\Phi(\lambda)=(\mu,\nu)$, we have $|\lambda|=|\mu|+|\nu|$,  $\ell(\lambda)=\ell(\mu)+\ell(\nu)$, and the colour sequence of $\lambda$ restricted to the colours in $\Co_{\mathrm{sup}}\sqcup \Co_{\mathrm{inf}}$ is the same as the colour sequence of $\mu$ restricted to the colours in $\Co_{\mathrm{sup}}\sqcup \Co_{\mathrm{inf}}$.
\end{theo}

We prove \Thm{theo:dualitycapaprimc} bijectively, by inserting the parts of $\nu$ into the generalised coloured partition $\mu$. In Section \ref{sec:insertion}, we explain the insertion process and the properties it satisfies. In Section \ref{sec:bijectivemaps}, we explicitly give the bijection $\Phi$. In Section \ref{sec:applicationtocn1}, we apply \Thm{theo:dualitycapaprimc} to the grounded partitions related to our $1$ perfect crystal $\B$ of type $C_n^{(1)}$ and prove Theorem \ref{th:bijCn}.

\subsection{Insertion of parts}
\label{sec:insertion}

We start by defining the notion of insertion. Let $f\in \Co_{\mathrm{free}}$, $p\in \Z$ and $\pi=(\pi_0,\ldots,\pi_{s-1})$ a generalised coloured partition with relation $\gg_\ep$. Inserting $p_f$ into $\pi$ consists in transforming $\pi$ into $\tilde{\pi}=(\pi_0,\ldots,\pi_{j-1},p_f,\pi_j,\ldots,\pi_{s-1} )$ for some $j\in\{0,\ldots,s\}$ in such a way that $\tilde{\pi}$ is a generalised coloured partition with relation $\gg_\ep$. 
Inserting a part into a pattern is the same as inserting the part into the generalised coloured partition consisting exactly of the parts of the pattern. 
In particular, when $s=1$ and $j=0$, we say that say that we insert $p_f$ to the left of $\pi_0$, and when $s=1$ and $j=1$, we say that say that we insert $p_f$ to the right of $\pi_0$.

First we study the patterns where all the parts have the same size, using $(1)$--$(4)$ of Definition \ref{def:welldef}. The following proposition, which follows directly from the definition of $\ep$, gives an exhaustive list of such patterns.

\begin{prop}
\label{prop:seq0000}
A sequence $p_{c_1}, \ldots , p_{c_s}$ with $p \geq 0$ is is a pattern if and only if
$c_1, \dots, c_s$ is a sequence of colours such that for all $i \in \{1 , \dots, s-1\},$ $\ep(c_i,c_{i+1})=0$. In this case, there exist unique integers $1\leq u \leq v \leq s+1$ such that
\begin{align}
 \{c_1,\ldots,c_{u-1}\}&\subset \Co_{\mathrm{sup}}\,,\nonumber\\
 c_u=\cdots =c_{v-1}&\in \Co_{\mathrm{free}}\,,\\
 \{c_v,\ldots,c_{s}\}&\subset \Co_{\mathrm{inf}}\nonumber\,,
\end{align}
with the convention that $\{c_a,\ldots,c_b\}=\emptyset$ if $a>b$.
  
\end{prop}

\subsubsection{Insertion between two parts}
\label{sec:insertbetweenpair}

In this section, we study all the possible insertions of a part $p_f$ with colour $f\in \Co_{\mathrm{free}}$ between the two parts of a pattern $p^{(1)}_{c_1},p^{(2)}_{c_2}$ such that $p\in \{p^{(1)},p^{(2)}\}$.

Observe that a part $p_f$ can be inserted between the two parts of a pattern $p^{(1)}_{c_1},p^{(2)}_{c_2}$ with
$p=p^{(1)}$ and $c_1\in \Co_{\mathrm{free}}$ (resp. $p=p^{(2)}$ and $c_2\in \Co_{\mathrm{free}}$) iff $f=c_1$ (resp. $f=c_2$). 
The only remaining cases to study are insertions such that $p=p^{(1)}$ and $c_1\in \Co_{\mathrm{sup}}$ or $p=p^{(2)}$ and $c_2\in \Co_{\mathrm{inf}}$.

We start with the case $p=p^{(1)}=p^{(2)}$, which corresponds to a pattern of the form $p_{c_1},p_{c_2}$. By Proposition \ref{prop:seq0000}, inserting $p_f$ between $p^{(1)}_{c_1}$ and $p^{(2)}_{c_2}$ is not possible when $c_1$ and $c_2$ are either both in
$\Co_{\mathrm{sup}}$ or both in $\Co_{\mathrm{inf}}$. The following propositions deals with the remaining case, i.e. $(c_1,c_2)\in \Co_{\mathrm{sup}}\times\Co_{\mathrm{inf}}$.

\begin{prop}
\label{lem:seq00}
 For any pair $(c_1,c_2)\in \Co_{\mathrm{sup}}\times\Co_{\mathrm{inf}}$ such that $\ep(c_1,c_2)=0$, a part $p_f$ with $f\in \Co_{\mathrm{free}}$ can be inserted between the parts $p_{c_1},p_{c_2}$ to obtain
 $$p_{c_1},p_f,p_{c_2}$$
 if and only if $\ep(c_1,f)=\ep(f,c_2)=0$.
\end{prop}

We now study the case where $p^{(1)}\neq p^{(2)}$. This necessarily means that $p^{(1)}>p^{(2)}$. We start with the case where $p=p^{(1)}$ and $c_1\in \Co_{\mathrm{sup}}$.

\begin{prop}
\label{lem:seq10up}
For $(c_1,f)\in \Co_{\mathrm{sup}}\times\Co_{\mathrm{free}}$, the exhaustive list of possible insertions of $p_f$ between parts $p_{c_1},p^{(2)}_{c_2}$ is as follows: 
\begin{enumerate}
\item for any colour $c_2$ in $\Co\sqcup \{c_\infty\}$, and any integer $u\geq 2$, we can insert a part $p_f$ between the parts of the pattern $p_{c_1},(p-u)_{c_2}$ to obtain
$$p_{c_1},p_f,(p-u)_{c_2}$$
if and only if $\ep(c_1,f)=0$,
\item for any colour $c_2 \in \Co_{\mathrm{free}}\sqcup \Co_{\mathrm{inf}}\sqcup \{c_\infty\}$, we can insert a part $p_f$ between the parts of the pattern $p_{c_1},(p-1)_{c_2}$ to obtain
$$p_{c_1},p_f,(p-1)_{c_2}$$
if and only if $\ep(c_1,f)=0$,
\item for any colour $c_2 \in \Co_{\mathrm{sup}}$ such that $\ep(c_1,c_2)\in \{0,1\}$, 
we can insert a part $p_f$ between the parts of the pattern $p_{c_1},(p-1)_{c_2}$ to obtain
$$p_{c_1},p_f,(p-1)_{c_2}$$
if and only if $\ep(c_1,f)=0$ and $\ep(f,c_2)=1$.
\end{enumerate}
\end{prop}

We conclude with the case where $p=p^{(2)}$ and $c_2\in \Co_{\mathrm{sup}}$.

\begin{prop}
\label{lem:seq10low}
For $(c_2,f)\in \Co_{\mathrm{inf}}\times\Co_{\mathrm{free}}$, the exhaustive list of possible insertions of $p_f$ between parts $p^{(1)}_{c_1},p_{c_2}$ is as follows: 
\begin{enumerate}
\item for any colour $c_1$ in $\Co$, and any integer $u\geq 2$, we can insert a part $p_f$ between the parts of the pattern $(p+u)_{c_1},p_{c_2}$ to obtain
$$(p+u)_{c_1},p_f,p_{c_2}$$
if and only if $\ep(c_1,f)=0$,
\item for any colour $c_1 \in \Co_{\mathrm{sup}}\sqcup \Co_{\mathrm{free}}$, we can insert a part $p_f$ between the parts of the pattern $(p+1)_{c_1},p_{c_2}$ to obtain
$$(p+1)_{c_1},p_f,p_{c_2}$$
if and only if $\ep(f,c_2)=0$,
\item for any colour $c_1 \in \Co_{\mathrm{inf}}$ such that $\ep(c_1,c_2)\in \{0,1\}$, 
we can insert a part $p_f$ between the parts of the pattern $(p+1)_{c_1},p_{c_2}$ to obtain
$$(p+1)_{c_1},p_f,p_{c_2}$$
if and only if $\ep(c_1,f)=1$ and $\ep(f,c_2)=0$.
\end{enumerate}
\end{prop}

The set of admissible colours $f$ for the part $p_f$ possibly inserted between $p^{(1)}_{c_1}$ and $p^{(2)}_{c_2}$ depends on both colours $c_1$ and $c_2$ only if $(c_1,c_2) \in\{(c,c')\in \Co_{\mathrm{sup}}\times \Co_{\mathrm{inf}}: \ep(c,c')=0\}\sqcup \{(c,c')\in \Co_{\mathrm{sup}}^2: \ep(c,c')\in \{0,1\}\}\sqcup \{(c,c')\in \Co_{\mathrm{inf}}^2: \ep(c,c')\in \{0,1\}\}.$
In that case, it is in particular possible to insert a part with colour $f=\gamma(c_1,c_2)$.
Otherwise, when the colour $f$ only depends on the colour $c$ of the part $p_c$ with the same size as $p_f$, it is possible to insert $p_f$ with $f=\delta(c)$. These particular choices for $f$ allow us to forbid, for all the pairs $p^{(1)}_{c_1},p^{(2)}_{c_2}$ of consecutive parts, a unique insertion, giving rise to the forbidden patterns in $_{\delta,\gamma}^{c_0}\Pp_{\ep}^{c_\infty}$ (see Definition \ref{def:Pcomplique}).

\begin{rem}
For any  $(c_1,c_2)\in \Co_{\mathrm{sup}}\times \Co_{\mathrm{inf}}$, we can insert between the parts $(p+1)_{c_1},p_{c_2}$ two parts  $(p+1)_{f_1}$ and $p_{f_2}$ with $f_1,f_2$ in $\Co_{\mathrm{free}}$ if only if $\ep(c_1,f_1)=\ep(f_2,c_2)=0$. The admissible colours $f_1$ depend only on $c_1$, and the allowed colours $f_2$ depend only on $c_2$.  
\end{rem}

\subsubsection{Insertion at the extremities}

Recall that 
$$
 \ep(\Co_{\mathrm{free}},c_\infty) = \{1\}\,,\,
 \ep(\Co_{\mathrm{inf}},c_\infty) \subset \{1,2\}\text{ and }
  \ep(\Co_{\mathrm{sup}},c_\infty) \subset  \{0,1\}\,.
$$
Then, similarly to \Prp{prop:seq0000}, if a partition in $\Pp_{\ep}^{c_\infty}$ has several parts of size $0$, then these parts are of the form 
$$0_{c_1},\ldots,0_{c_s},0_{c_\infty}$$
with $c_1,\ldots,c_s \in \Co_{\mathrm{sup}}$. Thus it is not possible to insert a part $0_f$ with $f\in \Co_{\mathrm{free}}$.
We now study the insertion of $1_f$ for some $f\in \Co_{\mathrm{free}}$ at the end of the partition.

\begin{itemize}

 \item When the partition ends with $1_{c},0_{c_\infty}$ with $c\in \Co_{\mathrm{sup}}$, the insertion of $1_f$ to the right of $1_c$ is allowed if and only if $\ep(c,f)=0$.
 
 \item When the partition ends with $1_{c},0_{c_\infty}$ with $c\in \Co_{\mathrm{inf}}$, if the insertion of $1_f$ is allowed, it will always be to the left of $1_c$. Thus $1_c$ remains the last part before $0_{c_\infty}$.
  
\end{itemize}

We finally study the case of an insertion at the beginning of the partition.

\begin{itemize}
\item When the first part is $p_{c}$ with $c\in \Co_{\mathrm{sup}}$, for any free colour $f$,
if the insertion of $p_f$ is allowed, it will always be to the right of $p_c$. Thus $p_c$ remains the first part of the partition.
 
 \item When the first part is $p_{c}$ with $c\in \Co_{\mathrm{inf}}$, 
the insertion of $p_f$ to the leftt of $p_c$ is allowed if and only if $\ep(f,c)=0$.
 \end{itemize}

Thus we can extend Propositions \ref{lem:seq00}, \ref{lem:seq10up} and \ref{lem:seq10low} as follows to include the insertions at the extremities:

\begin{itemize}
 \item the insertion of $p_f$ at the end of the partition corresponds to inserting it between $p_{c_1}$ and $0_{c_\infty}$ with $c_1\in \Co_{\mathrm{sup}}$ and $p\geq 1$,
 
 \item the insertion of $p_f$ at the beginning of the partition corresponds to inserting it between $\infty$ and $p_{c_2}$ with $c_2\in \Co_{\mathrm{inf}}$ and $p\geq 1$ (here $\infty$ is a virtual part, not considered to belong to the partition).
 \end{itemize}

\subsection{Bijective proof of \Thm{theo:dualitycapaprimc}}
\label{sec:bijectivemaps}

\subsubsection{The map $\Phi$}

Let us consider a partition $\lambda\in \mathcal \Pp_{\ep}^{c_\infty}$. We want to build $\Phi(\lambda)=(\mu,\nu)\in \,_{\delta,\gamma}^{c_0}\Pp_{\ep}^{c_\infty} \times \Pp$. 
First, recall that $\ep(c,c_0) = \ep(c_0,c) = \chi(c\neq c_0)$ for any colour $c\in \Co$.
 This is equivalent to saying that, in $\lambda$,  the parts coloured by $c_0$ have sizes different from all the parts whose colour is not $c_0$.
We first consider $\nu$ to be the empty partition. Then we transform some parts $p_f$ with $f\in \Co_{\mathrm{free}}$ into colourless parts $p$ and insert them into $\nu$ as follows.

\begin{enumerate}

\item Take all the parts of $\lambda$ with colour $c_0$ and move them to $\nu$ (after removing their colour $c_0$).
Since the parts which were to the left and to the right of the $p_{c_0}$'s in $\lambda$
have respectively sizes greater and smaller than $p$, this means that their sizes differ by at least $2$. 
Given that $\ep(\Co , \Co \sqcup \{c_\infty\})\subset \{0,1,2\}$, by removing the parts with colour $c_0$ from $\la$,
the remaining partition $\lambda'$ is still in $\Pp_{\ep}^{c_\infty}$. Furthermore, the parts of $\lambda'$ have sizes different from 
the sizes of the parts of $\nu$.

\item For all the parts $p_f$ in $\lambda'$ with $f\in \Co_{\mathrm{free}}\setminus \{c_0\}$ which appear at least twice, transform all the occurrences of $p_f$ but one into $p$ and move them to $\nu$. 
Since one occurrence remains for all such parts, we obtain a partition $\lambda''$ that is still in $\Pp_{\ep}^{c_\infty}$, and has no repeated parts $p_f$ with free colours and no part coloured $c_0$.
Also, the only parts of $\lambda''$ having the same size as some part of $\nu$
are coloured with $f\in \Co_{\mathrm{free}}\setminus \{c_0\}$.

\item For all the parts $p_f$ that appear in patterns $p^{(1)}_{c_1},p_f,p^{(2)}_{c_2}$ of $\lambda''$ which are forbidden in $_{\delta,\gamma}^{c_0}\Pp_{\ep}^{c_\infty}$, transform the parts $p_f$ into $p$ and move them to $\nu$. 
Note that such parts $p_f$ may have been among those which were repeated in Step (2), and can only appear in forbidden patterns with $p=p^{(1)}$ and $c_1 \in \Co_{\mathrm{sup}}$, or 
$p=p^{(2)}$ and $c_2 \in \Co_{\mathrm{inf}}$. One can also observe that, when removing $p_f$ from such patterns, the remaining patterns $p^{(1)}_{c_1},p^{(2)}_{c_2}$ are always allowed in $_{\delta,\gamma}^{c_0}\Pp_{\ep}^{c_\infty}$.
At the end of this step, the partition obtained from $\lambda''$ does not contain any forbidden pattern nor any part coloured $c_0$, and the parts $p_f$ with $f\in \Co_{\mathrm{free}}$ cannot be repeated. We set this partition to be $\mu$.
\end{enumerate}
At the end, we obtain a pair of partitions $(\mu,\nu)\in\, _{\delta,\gamma}^{c_0}\Pp_{\ep}^{c_\infty}\times \Pp$ as desired.

\begin{rem}
The only parts in $\nu$ which do not have the same size as some part in $\mu$ are those coming from the parts of $\lambda$ with colour $c_0$ in Step (1).
\end{rem}

\subsubsection{The map $\Phi^{-1}$}

We now describe the inverse map $\Phi^{-1}$. Let us start with $(\mu,\nu)\in\, _{\delta,\gamma}^{c_0}\Pp_{\ep}^{c_\infty}\times \Pp$, and insert the parts $p$ of $\nu$ into the partition $\mu$ as follows (note that at Step (1), we only insert one copy of each part):

\begin{enumerate}

\item If $\mu$ does not contain any part $p_f$ with $f\in \Co_{\mathrm{free}}\setminus \{c_0\}$, but contains a part $p_c$ with $c \in \Co_{\mathrm{sup}}\sqcup \Co_{\mathrm{inf}}$, we proceed as follows.

\begin{itemize}

\item If there is a pair of colours $(c_1,c_2)\in \Co_{\mathrm{sup}}\times\Co_{\mathrm{inf}}$ such that the pattern $p_{c_1},p_{c_2}$ is in $\mu$, then necessarily $\ep(c_1,c_2)=0$.
By \Prp{prop:seq0000}, there is only one such pair among the parts of size $p$.
Set $f=\gamma(c_1,c_2)$, transform  the part $p$ into $p_f$, and insert $p_f$ between $p_{c_1}$ and $p_{c_2}$. By Proposition \ref{lem:seq00}, we obtain the pattern 
$$p_{c_1},p_{\gamma(c_1,c_2)},p_{c_2}$$
which is forbidden in $ _{\delta,\gamma}^{c_0}\Pp_{\ep}^{c_\infty}$. 

Note that this is the only suitable insertion of a part of size $p$ to ensure that $\Phi^{-1}$ is well-defined and indeed the inverse bijection of $\Phi$, as we now show.
\begin{enumerate}
\item It is not possible to insert a part $p_f$ with $f\in \Co_{\mathrm{free}}$ in the sequence to the left of $p_{c_1}$ by Proposition \ref{prop:seq0000}.
  
\item Similarly, we cannot insert a part $p_f$ with $f\in \Co_{\mathrm{free}}$ in the sequence to the right of $p_{c_2}$.

\item Finally, inserting a part $p_f$ with $f\neq \gamma(c_1,c_2)$ and $\ep(c_1,f)=\ep(f,c_2)=0$ into the pair $p_{c_1},p_{c_2}$ would be allowed according to the definition of $\epsilon$, but the pattern
$p_{c_1},p_f,p_{c_2}$
is allowed in $_{\delta,\gamma}^{c_0}\Pp_{\ep}^{c_\infty}$, and this insertion would make the map $\Phi^{-1}$ non-injective.

\end{enumerate}

\item If all the parts of size $p$ in $\mu$ have colours in $\Co_{\mathrm{sup}}$, denote by $c_1$ the colour of the last of these parts. With the same reasoning as above, we cannot insert a part $p_f$ with $f\in \Co_{\mathrm{free}}$ in the 
sequence to the left of $p_{c_1}$. 
Note that the part to the right of $p_{c_1}$ has necessarily a size less than $p$. Using Proposition \ref{lem:seq10up}, insert $p$ in the following way:

\begin{enumerate}

\item If the part to the right of $p_{c_1}$ has size less than $p-1$, transform $p$ into $p_{\delta(c_1)}$ and  insert it to the right of $p_{c_1}$. We obtain the pattern
$$p_{c_1},p_{\delta(c_1)},(p-u)_{c_2}$$ for some integer $u \geq 2 $, which is forbidden in $ _{\delta,\gamma}^{c_0}\Pp_{\ep}^{c_\infty}$.

\item If the part to the right of $p_{c_1}$ has size $p-1$ and colour $c_2\in (\Co_{\mathrm{free}}\setminus \{c_0\} )\sqcup \Co_{\mathrm{inf}}\sqcup\{c_\infty\}$,  transform $p$ into $p_{\delta(c_1)}$ and insert it to the right of $p_{c_1}$. We obtain the pattern
$$p_{c_1},p_{\delta(c_1)},(p-1)_{c_2}$$ which is forbidden in $ _{\delta,\gamma}^{c_0}\Pp_{\ep}^{c_\infty}$.

\item If the part to the right of $p_{c_1}$ has size $p-1$ and colour $c_2\in \Co_{\mathrm{sup}}$, then necessarily $\ep(c_1,c_2)\in \{0,1\}$.
In that case, transform $p$ into $p_{\gamma(c_1,c_2)}$ and insert it to the right of $p_{c_1}$. We obtain the pattern
$$p_{c_1},p_{\gamma(c_1,c_2)},(p-1)_{c_2}$$ which is forbidden in $_{\delta,\gamma}^{c_0}\Pp_{\ep}^{c_\infty}$.

\end{enumerate}

\item We finish with the remaining case, where all the parts $p_c$ in $\mu$ are such that $c\in \Co_{\mathrm{inf}}$. Let $c_2$ be the colour of the first part of size $p$.
It is not possible to insert $p_f$ with $f \in \Co_{\mathrm{free}}$ in the sequence to the right of $p_{c_2}$ by Proposition \ref{prop:seq0000}. 
Also, the part to the left of $p_{c_2}$, if it exists, has necessarily a size greater than $p$. Using Proposition \ref{lem:seq10low}, insert $p$ in the following way:

\begin{enumerate}

\item If there is no part to the left of $p_{c_2}$, transform the part $p$ into $p_{\delta(c_2)}$ and insert it to the left of $p_{c_2}$. We obtain the pattern
$$p_{\delta(c_2)},p_{c_2}$$
which is forbidden in $_{\delta,\gamma}^{c_0}\Pp_{\ep}^{c_\infty}$.

\item If the part to the left of $p_{c_2}$ has a size greater than $p+1$, transform the part $p$ into $p_{\delta(c_2)}$ and insert it to the left of $p_{c_2}$. We obtain the pattern
$$(p+u)_{c_1},p_{\delta(c_2)},p_{c_2}$$ 
for some integer $u \geq 2 $,
which is forbidden in $_{\delta,\gamma}^{c_0}\Pp_{\ep}^{c_\infty}$.

\item If the part to the left of $p_{c_2}$ has size $p+1$ and colour $c_1 \in (\Co_{\mathrm{free}}\setminus \{c_0\})\sqcup \Co_{\mathrm{sup}}$, transform the part $p$ into $p_{\delta(c_2)}$ and insert it to the left of $p_{c_2}$. We obtain the pattern
$$(p+1)_{c_1},p_{\delta(c_2)},p_{c_2}$$ which is forbidden in $_{\delta,\gamma}^{c_0}\Pp_{\ep}^{c_\infty}$.

\item If the part to the left of $p_{c_2}$ has size $p+1$ and colour $c_1 \in \Co_{\mathrm{inf}}$, then necessarily $\ep(c_1,c_2)\in \{0,1\}$. Transform the part $p$ into $p_{\gamma(c_1,c_2)}$ and insert it to the left of $p_{c_2}$. We obtain the pattern
$$(p+1)_{c_1},p_{\gamma(c_1,c_2)},p_{c_2}$$ which is forbidden in $_{\delta,\gamma}^{c_0}\Pp_{\ep}^{c_\infty}$.

\end{enumerate}

\end{itemize}
Let $\mu'$ denote the resulting partition.

The order in which the parts $p_f$ were inserted does not matter. Indeed, when the insertion between a part of colour $c_1$ and a part of colour $c_2$ depends both on $c_1$ and $c_2$, then there is only one possibility for the inserted part $p_f$. And in the other cases, the insertion of a part of size $p$ does never interfere with the insertion of a part of another size.

Moreover, for the above cases, which form an exhaustive list of insertions $p_f$ into a pair $p^{(1)}_{c_1},p^{(2)}_{c_2}$ with $p=p^{(1)}$ and $c_1 \in \Co_{\mathrm{sup}}$, or 
$p=p^{(2)}$ and $c_2 \in \Co_{\mathrm{inf}}$, the choice of colour $f$ so that the obtained pattern $p^{(1)}_{c_1},p_f,p^{(2)}_{c_2}$ is forbidden in $_{\delta,\gamma}^{c_0}\Pp_{\ep}^{c_\infty}$ is unique.
The partition $\mu'$ has some forbidden patterns of $_{\delta,\gamma}^{c_0}\Pp_{\ep}^{c_\infty}$, no repeated part $p_f$ with $f \in \Co_{\mathrm{free}}$ and no part coloured by $c_0$. Thus Step (1) is the exact inverse of Step (3) of $\Phi$.

\item If there is a part $p_f$ in $\mu'$ with $f\in \Co_{\mathrm{free}}\setminus \{c_0\}$, transform all the parts $p$ into $p_f$ and insert them into $\mu'$. We obtain a partition $\mu''$ with some forbidden patterns of $_{\delta,\gamma}^{c_0}\Pp_{\ep}^{c_\infty}$ and repeated parts $p_f$ with $f\in \Co_{\mathrm{free}}\setminus \{c_0\}$, but no part coloured $c_0$. This is the inverse of Step (2) of $\Phi$.

\item After Steps (1) and (2), the only remaining parts in $\nu$ are such that there is no part in $\mu''$ with the same size. We transform these parts $p$ into $p_{c_0}$ and insert them into $\mu''$. 
The resulting partition has some forbidden patterns of $_{\delta,\gamma}^{c_0}\Pp_{\ep}^{c_\infty}$, repeated parts $p_f$ with $f\in \Co_{\mathrm{free}}$, and parts coloured $c_0$. Set this partition to be $\lambda$. This is the exact inverse of the Step (1) of $\Phi$.

\end{enumerate}

The partition $\lambda$ is a partition of $\Pp_{\ep}^{c_\infty}$, and we set $\Phi^{-1}(\mu,\nu)=\la$.

\subsection{Application to the level one perfect crystal of $C_n^{(1)}$} \label{sec:applicationtocn1}
Recall that we consider the crystal of Figure \ref{fig:crystalb}, with vertex set $$\B = \{\emptyset \} \sqcup \{(x,y) : x\leq y \in \overline{[n]}\},$$ and energy function $H$ given in Theorem \ref{th:energy}.

Let $\Co=\{c_b:b\in \B\}$ be the set of colours with indices in $\B$, where we write $c_{x,y}$ instead of $c_{(x,y)}$. Define
$\Co=\Co_{\mathrm{sup}}\sqcup \Co_{\mathrm{free}} \sqcup \Co_{\mathrm{inf}}$ with 
\begin{align*}
\Co_{\mathrm{free}} &= \{c_\emptyset,c_{x,\overline{x}}: x\in \{1,\ldots,n\}\},\\
\Co_{\mathrm{sup}} &= \{c_{x,y}: \overline{y}<x\leq y\},\\
\Co_{\mathrm{inf}} &= \{c_{x,y}: x\leq y<\overline{x}\}.
\end{align*}
Let $c_{\infty}$ be an additional colour and $\ep$ be a function on $\Co \times (\Co \sqcup \{c_{\infty}\} )$ such that $\ep(c_{b'},c_{b}):=H(b\ot b')$ for all $b,b'\in \Co$ and $\epsilon$ satisfies Condition (7) of Definition \ref{def:welldef}. Thus in particular $\ep$ takes values in $\{0,1,2\}$.

Define $\delta$ and $\gamma$ in the following way:
$$\delta(c_{x,y}):= \begin{cases} c_{\overline{y},y} &\text{ for all } c_{x,y} \in \Co_{\mathrm{sup}},
\\c_{x,\overline{x}}  &\text{ for all } c_{x,y} \in \Co_{\mathrm{inf}},
\end{cases}$$

$$
\gamma(c_{x,y},c_{x',y'}) := \begin{cases} c_{z,\overline{z}} &\text{ for all } (c_{x,y},c_{x',y'}) \in \Co_{\mathrm{sup}}\times \Co_{\mathrm{inf}} \text{ such that } \ep(c_{x,y},c_{x',y'})=0,
\\c_{\overline{y},y} &\text{ for all } (c_{x,y},c_{x',y'}) \in \Co_{\mathrm{sup}}^2 \text{ such that } \ep(c_{x,y},c_{x',y'})\in \{0,1\},
\\c_{x',\overline{x'}} &\text{ for all } (c_{x,y},c_{x',y'}) \in \Co_{\mathrm{inf}}^2 \text{ such that } \ep(c_{x,y},c_{x',y'})\in \{0,1\},
\end{cases}$$
where $z=\max\{x',\overline{y}\}$. 

\begin{lem}\label{lem:epwelldef}
The function $\ep$ is well-defined according to the decomposition $\Co_{\mathrm{sup}}\sqcup \Co_{\mathrm{free}} \sqcup \Co_{\mathrm{inf}}$ and $c_{\infty}$. Moreover the functions $\delta$ and $\gamma$ defined above satisfy Conditions (2)-(6) from Definition \ref{def:welldef}.
\end{lem}

\begin{proof}
The function $\ep$ satisfies the following.

\begin{enumerate}

 \item For all $x, y\in\{1,\ldots,n\}$ with $x\neq y$,
 $$\begin{cases}
 \ep(c_\emptyset,c_\emptyset)=0,\\
 \ep(c_\emptyset,c_{x,\overline{x}})=\ep(c_{x,\overline{x}},c_\emptyset)=1,\\
 \ep(c_{x,\overline{x}},c_{x,\overline{x}})=0,\\
 \ep(c_{x,\overline{x}},c_{y,\overline{y}})=1.
 \end{cases}$$
 
 \item For all $\overline{y}<x\leq y$, $\ep(c_{x,y},c_\emptyset)=\ep(c_\emptyset,c_{x,y})=1$. Let $z\in \{1,\ldots,n\}$. By \eqref{eq:h=2}, we have $\ep(c_{x,y},c_{z,\overline{z}})\neq 2$ because
 $y\geq \overline{n}>n\geq z$. Moreover, $\ep(c_{z,\overline{z}},c_{x,y})\neq 0$ since
 
 \begin{enumerate}
 \item we do not have \eqref{eq:asrho} which is equivalent to $x<z<\overline{y}$,
 \item we do not have \eqref{eq:inf} which is equivalent to $\overline{y}\geq x = z$,
  \item we do not have \eqref{eq:sup} which is equivalent to $\overline{y}=x = z$.
 \end{enumerate}
 
Finally, by \eqref{eq:sup}, $\ep(c_{x,y},\delta(c_{x,y}))=\ep(c_{x,y},c_{\overline{y},y})=0$.

   \item For all $x\leq y<\overline{x}$, $\ep(c_{x,y},c_\emptyset)=\ep(c_\emptyset,c_{x,y})=1$. Let $z\in \{1,\ldots,n\}$. Hence, by \eqref{eq:h=2},  $\ep(c_{z,\overline{z}},c_{x,y})\neq 2$ because
 $\overline{z}\geq \overline{n}>n\geq x$. Moreover, $\ep(c_{x,y},c_{z,\overline{z}})\neq 0$ since
 
 \begin{enumerate}
 \item we do not have \eqref{eq:asrho} which is equivalent to $\overline{y}<z<x$,
 \item we do not have \eqref{eq:inf} which is equivalent to $z=\overline{y}=x$,
  \item we do not have \eqref{eq:sup} which is equivalent to $z=\overline{y}\leq x$.
 \end{enumerate}
 
Finally, by \eqref{eq:inf}, $\ep(\delta(c_{x,y}),c_{x,y})=\ep(c_{x,\overline{x}},c_{x,y})=0$.

 \item For all $\overline{y}<x\leq y$ and $x'\leq y'<\overline{x'}$, by \eqref{eq:h=2}, $\ep(c_{x,y},c_{x',y'})\neq 2$ because
 $y\geq \overline{n}>n\geq x'$. Moreover, $\ep(c_{x',y'},c_{x,y})\neq 0$ since 
 
  \begin{enumerate}
 \item we do not have \eqref{eq:asrho} which is equivalent to $\overline{y'}<\overline{y}<x<x'$,
 \item we do not have \eqref{eq:inf} as $\overline{y}'\neq x'$,
  \item we do not have \eqref{eq:sup} as $\overline{y}\neq  x $.
 \end{enumerate}
 
Finally, by \eqref{eq:asrho},\eqref{eq:inf} and \eqref{eq:sup}, $\ep(c_{x,y},c_{x',y'})=0$ if and only if $x>x'$ and $y>y'$. In that case, setting $z=\max\{x',\overline{y}\}$,
 
\begin{itemize}
\item we first show that $\ep(c_{x,y},\gamma(c_{x,y},c_{x',y'}))=\ep(c_{x,y},c_{z,\overline{z}})=0$. If $\overline{y}<z=x'<x$, then $\ep(c_{x,y},c_{z,\overline{z}})=0$ by \eqref{eq:asrho}. Otherwise, $z=\overline{y}<x$ and $\ep(c_{x,y},c_{z,\overline{z}})=0$ by \eqref{eq:sup};

\item we then show that $\ep(\gamma(c_{x,y},c_{x',y'}),c_{x',y'})=\ep(c_{z,\overline{z}},c_{x',y'})=0$. If $x'<z=\overline{y}<\overline{y'}$, then  
$\ep(c_{z,\overline{z}},c_{x',y'})=0$ by \eqref{eq:asrho}. Otherwise, $z=x'<\overline{y'}$ and $\ep(c_{z,\overline{z}},c_{x',y'})=0$ by \eqref{eq:inf}.
\end{itemize}

\item For all $\overline{y}<x\leq y$, $\ep(c_{x,y},c_{\overline{y},y})=0$ by \eqref{eq:sup}.

\noindent For all  $\overline{y'}<x'\leq y'$, by \eqref{eq:h=2}, $\ep(c_{x,y},c_{x',y'})\neq 2$ if and only if $y>x'$. In that case, by \eqref{eq:h=2}, $\ep(c_{\overline{y},y},c_{x',y'})\neq 2$, and since  $\ep(c_{\overline{y},y},c_{x',y'})\neq 0$, we necessarily have $\ep(c_{\overline{y},y},c_{x',y'})=1$. Thus
$$\ep(c_{x,y},\gamma(c_{x,y},c_{x',y'}))+1=\ep(\gamma(c_{x,y},c_{x',y'}),c_{x',y'})=1\,.$$
   
\item For all $x'\leq y'<\overline{x'}$, by \eqref{eq:inf}, $\ep(c_{x',\overline{x'}},c_{x',y'})=0$.

\noindent For all $x\leq y<\overline{x}$, by
\eqref{eq:h=2}, $\ep(c_{x,y},c_{x',y'})\neq 2$ if and only if $y>x'$. In that case, by \eqref{eq:h=2}, $\ep(c_{x,y},c_{x',\overline{x'}})\neq 2$, and since  $\ep(c_{x,y},c_{x',\overline{x'}})\neq 0$, we necessarily have $\ep(c_{x,y},c_{x',\overline{x'}})=1$. Thus
$$\ep(c_{x,y},\gamma(c_{x,y},c_{x',y'}))=\ep(\gamma(c_{x,y},c_{x',y'}),c_{x',y'})+1=1\,.$$
\end{enumerate}
\end{proof}

In the following, set $\Sc=\Co\setminus\{c_\emptyset\}$.

\begin{lem}\label{lem:eprho}
Let $\rho$ be the function defined on $\Sc \times (\Sc \sqcup \{c_{\infty}\})$ by $\rho(\Sc,c_\infty)=\ep(\Sc,c_\infty)$ and for all $ x\leq y \in \overline{[n]}, x' \leq y' \in \overline{[n]}$, by
$$\rho(c_{x',y'},c_{x,y})=\chi(x\geq x')+\chi(y\geq y')-\chi(y\geq y'>x \geq x').$$ 

Let $\tilde{\Pp}_\rho^{c_\infty}$ be the set of partitions
$\pi=(\pi_0,\ldots,\pi_{s-1},\pi_s=0_{c_\infty})$ with colours in $\Sc \sqcup \{c_{\infty}\}$ such that for all $i\in \{0,\ldots,s-1\}$, 
$$c(\pi_i)\neq c_\infty \text{ and }|\pi_i|-|\pi_{i+1}|\geq \rho(c(\pi_i),c(\pi_{i+1})).$$
Then, setting $c_0=c_\emptyset$, we have
$$_{\delta,\gamma}^{c_0}\Pp_{\ep}^{c_\infty}=\tilde{\Pp}_\rho^{c_\infty}.$$ 

\end{lem}

\begin{proof}
We first observe that, by Remark \ref{rem:rk1}, $\rho(c_{x',y'},c_{x,y})=\ep(c_{x',y'},c_{x,y})$ except for the cases $\overline{y}\geq x = \overline{y'}=x'$ and $\overline{y}=x = \overline{y'}< x'$, where $\rho(c_{x',y'},c_{x,y})=\ep(c_{x',y'},c_{x,y})+1=1$. These exceptions can be divided in three disjoint cases:
$$
\begin{cases}
\overline{y}=x = \overline{y'}= x',\\
\overline{y}> x = \overline{y'}=x',\\
\overline{y}=x = \overline{y'}< x'.
\end{cases}
$$
Thus $\tilde{\Pp}_\rho^{c_\infty}$ is exactly exactly the set of $c_\emptyset$-regular partitions of $\Pp_{\ep}^{c_\infty}$ which avoid the patterns:
\begin{enumerate}
\item $p_{c_{x,\overline{x}}},p_{c_{x,\overline{x}}}$ for all $x\in \{1,\ldots,n\}$,
\item $p_{c_{x,\overline{x}}},p_{c_{x,y}}$ for all $x\leq y <\overline{x}$,
\item $p_{c_{x,y}},p_{c_{\overline{y},y}}$ for all $\overline{y}<x\leq y$.
\end{enumerate}
Since each forbidden pattern for the set $_{\delta,\gamma}^{c_0}\Pp_{\ep}^{c_\infty}$ contains at least one of the patterns (1)-(3), this implies that $\tilde{\Pp}_\rho^{c_\infty} \subset \,_{\delta,\gamma}^{c_0}\Pp_{\ep}^{c_\infty}$. 

We now prove that $_{\delta,\gamma}^{c_0}\Pp_{\ep}^{c_\infty}\subset \tilde{\Pp}_\rho^{c_\infty}$ by showing that the forbidden patterns of $\tilde{\Pp}_\rho^{c_\infty}$ are also forbidden in $_{\delta,\gamma}^{c_0}\Pp_{\ep}^{c_\infty}$.

\begin{enumerate}

\item For all $x\in \{1,\ldots,n\}$, the pattern $p_{c_{x,\overline{x}}},p_{c_{x,\overline{x}}}$ is forbidden in in $_{\delta,\gamma}^{c_0}\Pp_{\ep}^{c_\infty}$.

\item For all $x\leq y <\overline{x}$, we show that $p_{c_{x,\overline{x}}},p_{c_{x,y}}$ is forbidden in $_{\delta,\gamma}^{c_0}\Pp_{\ep}^{c_\infty}$. To do so, we first prove that all the patterns of the form $q_{c_{x',y'}},p_{c_{x,\overline{x}}},p_{c_{x,y}}$ allowed in $\Pp_{\ep}^{c_\infty}$ are forbidden in $_{\delta,\gamma}^{c_0}\Pp_{\ep}^{c_\infty}$.

\begin{enumerate}
\item If $q=p$, then $\ep(c_{x',y'},c_{x,\overline{x}})=0$, and either $\overline{y'}\leq  x < x'$ by the union of \eqref{eq:asrho} and \eqref{eq:sup}, or $\overline{y}= x=\overline{y'}=x'$ by \eqref{eq:inf}. In the first case, we have $x < x'$ and $y'\geq \overline{x}>y$, so that $\ep(c_{x',y'},c_{x,y})=0$ by \eqref{eq:asrho}. Since $\max\{\overline{y'},x\}=x$, we obtain the pattern $p_{c_{x',y'}},p_{\gamma(c_{x',y'},c_{x,y})},p_{c_{x,y}}$, which is forbidden in $_{\delta,\gamma}^{c_0}\Pp_{\ep}^{c_\infty}$. In the second case, the pattern $p_{c_{x,\overline{x}}},p_{c_{x,\overline{x}}},p_{c_{x,y}}$ is forbidden in $_{\delta,\gamma}^{c_0}\Pp_{\ep}^{c_\infty}$. 
\item If $q=p+1$, then $\ep(c_{x',y'},c_{x,\overline{x}})\leq 1$. If $\overline{y'}\leq  x'\leq y'$, the pattern $(p+1)_{c_{x',y'}},p_{c_{x,\overline{x}}},p_{c_{x,y}}$ is forbidden in $_{\delta,\gamma}^{c_0}\Pp_{\ep}^{c_\infty}$. Otherwise, $ x'\leq y'<\overline{x'}$. In that case, as $\ep(c_{x',y'},c_{x,\overline{x}})\neq 2$, by \eqref{eq:h=2}, $y'>x$ so that $\ep(c_{x',y'},c_{x,y})\neq 2$. Once again, the pattern $(p+1)_{c_{x',y'}},p_{c_{x,\overline{x}}},p_{c_{x,y}}$ is forbidden in $_{\delta,\gamma}^{c_0}\Pp_{\ep}^{c_\infty}$.
\item If $q\geq p+2$, the pattern $q_{c_{x',y'}},p_{c_{x,\overline{x}}},p_{c_{x,y}}$ is forbidden in $_{\delta,\gamma}^{c_0}\Pp_{\ep}^{c_\infty}$.
\end{enumerate}
The pattern $p_{c_{x,\overline{x}}},p_{c_{x,y}}$ is forbidden in $_{\delta,\gamma}^{c_0}\Pp_{\ep}^{c_\infty}$ at the beginning of the partition. Therefore, together with the cases above, we conclude that the pattern $p_{c_{x,\overline{x}}},p_{c_{x,y}}$ is simply forbidden in $_{\delta,\gamma}^{c_0}\Pp_{\ep}^{c_\infty}$. 
\item Similarly to what we did in $(2)$, we show that the pattern $p_{c_{x,y}},p_{c_{\overline{y},y}}$ is forbidden in $_{\delta,\gamma}^{c_0}\Pp_{\ep}^{c_\infty}$ for all $\overline{y}<x\leq y$.
\end{enumerate}
\end{proof}

Now we modify the minimal differences with the last part coloured $c_{\infty}$ in the functions $\epsilon$ and $\rho$ to make the connection with grounded partitions.

Let $b_0=\emptyset$, and for $i\in \{1,\ldots,n\}$, let  $b_i=(i,\overline{i})$. Then, for $i\in \{0,\ldots,n\}$ 
define the function $\epsilon_i$ on $\Co \times (\Co \sqcup \{c_{\infty}\})$ as follows:
$$\epsilon_i (c',c):= \begin{cases}
H(b \otimes b') &\text{ if } (c',c)=(c_{b'},c_b) \in \Co^2,\\
H(b_i \otimes b') &\text{ if } (c',c)=(c_{b'},c_{\infty}) \text{ with } c_{b'} \in \Co \setminus \{c_{b_i}\} ,\\
1 &\text{ if } c'=c_{b_i} \text{ and } c=c_{\infty}.
\end{cases}$$
The function $\rho_i$ is defined  on $\Sc \times (\Sc \sqcup \{c_{\infty}\})$ by $\rho_i(\Sc,c_\infty)=\ep_i(\Sc,c_\infty)$ and for all $ x\leq y \in \overline{[n]}, x' \leq y' \in \overline{[n]}$,
$$\rho_i(c_{x',y'},c_{x,y})=\chi(x\geq x')+\chi(y\geq y')-\chi(y\geq y'>x \geq x').$$ 

Applying Theorem \ref{theo:dualitycapaprimc} with $c_0 =c_{\emptyset}$ and combining it with Lemma \ref{lem:eprho}, we obtain the following.

\begin{cor}\label{cor:dualitycapaprimcCn}
For all $i\in \{0,\ldots,n\}$, there exists a bijection $\Phi$ between $\Pp_{\ep_i}^{c_\infty}$ and the product set 
 $\tilde{\Pp}_{\rho_i}^{c_\infty} \times \Pp$.
Furthermore, for $\Phi(\lambda)=(\mu,\nu)$, we have $|\lambda|=|\mu|+|\nu|$,  $\ell(\lambda)=\ell(\mu)+\ell(\nu)$, and the colour sequence of $\lambda$ restricted to the colours in $\Co_{\mathrm{sup}}\sqcup \Co_{\mathrm{inf}}$ is the same as the colour sequence of $\mu$ restricted to the colours in $\Co_{\mathrm{sup}}\sqcup \Co_{\mathrm{inf}}$.
\end{cor}

Let $\Pp_{\ep_i}$ (resp. $\Pp_{i,\rho}$) be the set obtained from $\Pp_{\ep_i}^{c_\infty}$ (resp. $\tilde{\Pp}_{\rho_i}^{c_\infty}$) by transforming the final part $0_{c_\infty}$ into $0_{c_{b_i}}$. Then $\Pp_{\ep_i}$ is exactly the set  $\Pp^{\gg}_{c_{b_i}}$of grounded partitions with ground $c_{b_i}$ and relation $\gg$.
Theorem \ref{th:bijCn} is proved.


\section{Connection with Frobenius partitions}
\label{sec:roadfrob}
Let $i\in \{0,\ldots,n\}$. In this section we show that the partitions of $\Pp_{i,\rho}$ can be identified with some simpler objects, namely the \textit{$C_n^{(1)}$-Frobenius partitions} defined in the introduction, which can be seen as coloured Frobenius partitions \cite{DK19} with additional interlacing conditions. We give a more general definition of these objects in Subsection \ref{sec:frob}, before explaining the correspondence with $\Pp_{i,\rho}$ in the following subsections.

\subsection{$C_n^{(1)}$-Frobenius partitions}
\label{sec:frob}
Let $(\Od,\geq)$ be an ordered set and let $\Rr=\{c_u: u\in \Od\}$ be a set of colours which we will call \textit{primary colours}. Recall that $>$ denotes the strict order corresponding to $\geq$, and 
extend the order $\geq$ on the set of primary-coloured integers $\Z_\Rr$ as follows:
\begin{equation}\label{eq:order}
k_{c_u}\geq l_{c_v}\Longleftrightarrow k-l\geq \chi(u<v)\,.
\end{equation}
Equivalently, $k_{c_u}> l_{c_v}$ if and only if $k-l\geq \chi(u\leq v)$.
This is also equivalent to ordering the coloured integers in the following way
$$\cdots  \leq (k-1)_{c_{u_1}} \leq (k-1)_{c_{u_2}} \leq\cdots \leq k_{c_{u_1}} \leq k_{c_{u_2}} \leq \cdots  \leq (k+1)_{c_{u_1}} \leq (k+1)_{c_{u_2}} \leq \cdots,$$
where $u_1 \leq u_2 \leq \cdots$ in $\Od$.

The definition of $C_n^{(1)}$-Frobenius partitions with relation $>$ and ground $k_c,l_d$ given in Definition \ref{def:grounded_frob} for the case $\Od = \ov{[n]}$ is still valid for general $\Od$. Hence we do not repeat it here.

Given a set of primary colours $\Rr$, define $\Sc$ the corresponding set of \textit{secondary colours}, which are commutative products of two primary colours, i.e. 
$$\Sc:=\{c_{x,y}=c_xc_y=c_yc_x:x\leq y \in \Od\}.$$ 
Let $\Z_\Sc$ denote the set of secondary-coloured integers.
In this section, if $k_{c_u}$ and $l_{c_v}$ are two primary-coloured integers, then we define their sum to be the secondary-coloured integer
$$k_{c_u}+l_{c_v}=(k+l)_{c_uc_v}.$$

Define the function $\rho$ on $\Sc^2$ by 
$$\rho(c_{x',y'},c_{x,y}):=\chi(x\geq x')+\chi(y\geq y')-\chi(y\geq y'>x \geq x').$$ 
Note that this $\rho$ has the same expression as the function $\rho$ from Lemma \ref{lem:eprho} restricted to $\Sc^2$.
The corresponding relation $\gg_\rho$ on $\Z_\Sc$ is given by 
$$k_c\gg_\rho l_d \text{ if and only if }k-l\geq \rho(c,d).$$

Finally, denote by $\Pp^{k_c+l_d}_\rho$ the set of generalised coloured partitions with parts in $\Z_\Sc$, order $\gg_\rho$, and last part equal to $k_c+l_d$. 
The main result of this section is a bijection between the coloured partitions of $\Pp^{k_c+l_d}_\rho$ and the $C_n^{(1)}$-Frobenius partitions of $\F^{k_c,l_d}_>$. But before being able to state this bijection, we need to introduce a decomposition of secondary coloured integers.

\subsection{Turning a coloured integer into a pair of coloured integers}
We decompose the coloured integers of $\Z_\Sc$ in the following way. For all $x\leq y\in \Od$ and $k\in \Z$, 
$$
\begin{cases}
2k_{c_{x,y}}=k_{c_y}+k_{c_x},\\
(2k+1)_{c_{x,y}}=(k+1)_{c_x}+k_{c_y}.
\end{cases}
$$
Denote respectively by $\eta$ and $\zeta$ the larger and smaller component in the decomposition above, namely
\begin{equation}\label{eq:etazeta}
\begin{cases}
(\eta(2k_{c_{x,y}}),\zeta(2k_{c_{x,y}})):=(k_{c_y},k_{c_x}),\\
(\eta((2k+1)_{c_{x,y}}),\zeta((2k+1)_{c_{x,y}})):=((k+1)_{c_x},k_{c_y}).
\end{cases}
\end{equation}
Notice that, by \eqref{eq:etazeta}, for $k\in \Z$ and $x\leq y\in \Od$, we have have the simple relations
\begin{equation}\label{eq:succession}
\begin{cases}
\eta(k_{c_{x,y}}+1)=\zeta(k_{c_{x,y}})+1,\\
\zeta(k_{c_{x,y}}+1)=\eta(k_{c_{x,y}}).
\end{cases}
\end{equation}
By \eqref{eq:order}, we deduce that
\begin{equation}\label{eq:encadreta}
\zeta(k_{c_{x,y}})\leq \eta(k_{c_{x,y}}) \leq \zeta(k_{c_{x,y}})+1.
\end{equation}
Conversely, if $k_{c_u}$ and $l_{c_v}$ are primary-coloured integers such that $l_{c_v} \leq k_{c_u} \leq l_{c_v}+1$, then they are respectively the $\eta$ and $\zeta$ component of a secondary-coloured integer, namely $(k+l)_{c_uc_v}$. Indeed, by \eqref{eq:order} we know that
$$\chi(u\leq v)\geq k-l \geq \chi(u<v),$$
and thus by \eqref{eq:etazeta}, 
$$(\eta((k+l)_{c_uc_v}),\zeta((k+l)_{c_uc_v}))=(k_{c_u},l_{c_v}).$$
Hence one can identify $\Z_\Sc$ with the set of pairs of primary-coloured integers $(k_{c_u},l_{c_v})$ such that $l_{c_v} \leq k_{c_u} \leq l_{c_v}+1$.

We now state the  bijection.
\begin{prop}\label{prop:frob}
The map 
$$(\pi_0,\ldots,\pi_{s-1},u+v)\mapsto((\eta(\pi_0),\ldots,\eta(\pi_{s-1}),u),(\zeta(\pi_0),\ldots,\zeta(\pi_{s-1}),v))$$
describes a bijection from $\Pp^{k_c+l_d}_\rho$ to $\F^{k_c,l_d}_>$ which preserves the size and colour sequence.
\end{prop}

To prove that it is indeed a bijection, we need the following lemma which relates the relation $\gg_\rho$ on $\Z_\Sc$ to the order $\geq$ on  $\Z_\Rr$ through the comparison of their $\eta$ and $\zeta$ components.

\begin{lem}\label{lem:secprim}
For all $x\leq y, x'\leq y'\in \Od$ and $k,l\in \Z$, we have 
\begin{equation}\label{eq:secprim}
k_{c_{x,y}}\gg_\rho l_{c_{x',y'}} \Longleftrightarrow \Big(\eta(k_{c_{x,y}})> \eta(l_{c_{x',y'}}) \text{ and }\zeta(k_{c_{x,y}})> \zeta(l_{c_{x',y'}})\Big).
\end{equation}
\end{lem}

\begin{proof}
By \eqref{eq:succession} and \eqref{eq:encadreta}, for all $k_{c_{x,y}} \in \Z_\Sc$,
\begin{equation}\label{eq:k+1}
\eta((k+1)_{c_{x,y}})\geq \eta(k_{c_{x,y}})\text{ and }\zeta((k+1)_{c_{x,y}})\geq \zeta(k_{c_{x,y}}).
\end{equation}
Hence, to show that 
$$k_{c_{x,y}}\gg_\rho l_{c_{x',y'}}\Longrightarrow \Big(\eta(k_{c_{x,y}})> \eta(l_{c_{x',y'}}) \text{ and }\zeta(k_{c_{x,y}})> \zeta(l_{c_{x',y'}})\Big),$$
it  suffices to show that 
\begin{equation}
\label{eq:a}
\eta(l_{c_{x,y}}+\rho(c_{x,y},c_{x',y'}))> \eta(l_{c_{x',y'}}) \text{ and }\zeta(l_{c_{x,y}}+\rho(c_{x,y},c_{x',y'}))> \zeta(l_{c_{x',y'}}).
\end{equation}
Similarly, showing that
$$k_{c_{x,y}}\not\gg_\rho l_{c_{x',y'}}\Longrightarrow \Big(\eta(k_{c_{x,y}})\leq \eta(l_{c_{x',y'}}) \text{ or }\zeta(k_{c_{x,y}})\leq \zeta(l_{c_{x',y'}})\Big)$$
is equivalent to showing that 
\begin{equation}
\label{eq:b}
\eta(l_{c_{x,y}}+\rho(c_{x,y},c_{x',y'})-1)\leq \eta(l_{c_{x',y'}}) \text{ or }\zeta(l_{c_{x,y}}+\rho(c_{x,y},c_{x',y'})-1)\leq \zeta(l_{c_{x',y'}}).
\end{equation}
Moreover, by \eqref{eq:succession}, the equivalence \eqref{eq:secprim} holds for $k,l$ if and only if it holds for $k+1,l+1$. Thus, without loss of generality, we may assume that $l$ is even, so that by setting $l=2m$, we have $\eta(l_{c_{x',y'}})=m_{c_{y'}}$ and $\zeta(l_{c_{x',y'}})=m_{c_{x'}}$.

We break the condition $x\leq y, x'\leq y'\in \Od$ into an exhaustive list of four cases.
\begin{enumerate}

\item If $x>x'$ and $y>y'$, then $\rho(c_{x,y},c_{x',y'})=0$, and by \eqref{eq:order}  and \eqref{eq:succession},
$$
\eta(l_{c_{x,y}})=m_{c_y}>m_{c_{y'}}=\eta(l_{c_{x',y'}})\text{ and }
\zeta(l_{c_{x,y}})=m_{c_x}>m_{c_{x'}} \zeta(l_{c_{x',y'}}),$$
so $\eqref{eq:a}$ is satisfied. Moreover,
$$\eta(l_{c_{x,y}}-1)=m_{c_x}\leq m_{c_{y'}}=\eta(l_{c_{x',y'}}),
$$
so $\eqref{eq:b}$ is true as well.
\item If $x\leq x'\leq y'<y$, then $\rho(c_{x,y},c_{x',y'})=1$, and by \eqref{eq:order} and \eqref{eq:succession},
$$
\eta(l_{c_{x,y}}+1)=(m+1)_{c_x}>m_{c_{y'}}\text{ and } 
\zeta(l_{c_{x,y}}+1)=m_{c_y}>m_{c_{x'}}.$$
so $\eqref{eq:a}$ is satisfied. Moreover,
$$
\zeta(l_{c_{x,y}})=m_{c_x}\leq m_{c_{x'}},
$$
proving \eqref{eq:b}.
\item If $x'<y\leq y'$, then $\rho(c_{x,y},c_{x',y'})=1$, and by \eqref{eq:order}  and \eqref{eq:succession},
$$
\eta(l_{c_{x,y}}+1)=(m+1)_{c_x}>m_{c_{y'}}\text{ and } 
\zeta(l_{c_{x,y}}+1)=m_{c_y}>m_{c_{x'}},$$
so $\eqref{eq:a}$ is true. And
$$
\eta(l_{c_{x,y}})=m_{c_y}\leq m_{c_{y'}},
$$
thus \eqref{eq:b} is satisfied.
\item If $x'\geq y$, then $\rho(c_{x,y},c_{x',y'})=2$, and by \eqref{eq:order}  and \eqref{eq:succession}, we have
$$
\eta(l_{c_{x,y}}+2)=(m+1)_{c_y}>m_{c_{y'}}\text{ and } 
\zeta(l_{c_{x,y}}+2)=(m+1)_{c_x}>m_{c_{x'}},$$
which proves \eqref{eq:a}. Also,
$$
\zeta(l_{c_{x,y}}+1)=m_{c_y}\leq m_{c_{x'}},
$$
so \eqref{eq:b} is true.
\end{enumerate}
In all the above cases, hence when $x\leq y, x'\leq y'$, \eqref{eq:secprim} is always true.
\end{proof}

The proof of Proposition \ref{prop:frob} follows immediately from \eqref{eq:encadreta} and   Lemma \ref{lem:secprim}.

\subsection{Minimal parts}\label{sec:minimal}
In this section we study the case $\Od=\overline{[n]}$.
The set of secondary colours $\Sc$ is the same as the one defined in Section \ref{sec:applicationtocn1}, and the function $\rho$ coincides to the one defined in Lemma \ref{lem:eprho} restricted to $\Sc^2$. We show that there is a bijection between the set $\Pp_{i,\rho}$ defined in Section \ref{sec:applicationtocn1} and certain $C_n^{(1)}$-Frobenius partitions.

\begin{theo}
\label{th:bijPrhoFrob}
The map 
$$(\pi_0,\ldots,\pi_{s-1},0_{c_{b_i}})\mapsto((\eta(\pi_0),\ldots,\eta(\pi_{s-1}),\eta(\omega_i)),(\zeta(\pi_0),\ldots,\zeta(\pi_{s-1}),\zeta(\omega_i)))$$
describes a bijection from $\Pp_{i,\rho}$ to $\F^{0_{c_{\ov{i+1}}},0_{c_i}}_>$ (resp. $\F^{0_{c_{\ov{1}}},-1_{c_{\ov{1}}}}_>$) for $i \in \{1, \dots , n \}$ (resp. $i=0$). Moreover this bijection preserves the size and colour sequence of all the parts except the last.
\end{theo}

To prove Theorem \ref{th:bijPrhoFrob}, we show that for all $i\in \{0,\ldots,n\}$, the set $\Pp_{i,\rho}$  can be identified with by $\Pp^{\omega_i}_\rho$, where
$$\omega_i := \begin{cases}
(-1)_{c_{\overline{1},\overline{1}}} &\text{ if } i=0,\\
0_{c_{i,\overline{i+1}}} &\text{ if } i\in \{1,\ldots,n\}.
\end{cases}
$$
Here we use the convention that $\overline{n+1}=n$.
Combined with Proposition \ref{prop:frob}, this correspondence establishes Theorem \ref{th:bijPrhoFrob} .

\begin{lem}
\label{lem:identification}
For $i\in \{0,\ldots,n\}$, the set $\Pp_{i,\rho}$ can be identified with the set $\Pp_\rho^{\omega_i}$ through the map
$$\phi:(\pi_0,\ldots,\pi_{s-1},0_{c_{b_i}})\mapsto (\pi_0,\ldots,\pi_{s-1},\omega_i).$$ 
\end{lem}
\begin{proof}
For all $\pi=(\pi_0,\ldots,\pi_{s-1},0_{c_{b_i}})\in \Pp_{i,\rho}$, we have:
\begin{itemize}
\item  $\pi_j\in \Z_\Sc$ for $0\leq j\leq s-1$,
\item $\pi_{j}\gg_\rho \pi_{j+1}$ for $0\leq j<s-1$,
\item $\pi_{s-1}=k_{c_{x,y}}$ with
\begin{equation}
\label{eq:star}
k\geq \begin{cases}
H(b_i \otimes b_{x,y}) &\text{ if }   b_{x,y}\neq c_{b_i},\\
1 &\text{ if }   b_{x,y}= c_{b_i}.
\end{cases}
\end{equation}
\end{itemize}
The map $\phi$ only modifies the last part. Hence we only need to show that 
\eqref{eq:star}
is equivalent to 
\begin{equation}
\label{eq:triangle}
\pi_{s-1}\gg_\rho\omega_i.
\end{equation}
We treat separately the cases $i=0$ and $i \in \{1, \dots , n \}$.
\begin{enumerate}
\item For $i=0$, \eqref{eq:star} simply becomes $k\geq 1$. By \eqref{eq:k+1}, $k\geq 1$ implies that
$$\eta(k_{c_{x,y}})\geq \eta(1_{c_{x,y}})=1_{c_x}\geq 1_{c_1} \text{ and }\zeta(k_{c_{x,y}})\geq \zeta(1_{c_{x,y}})=0_{c_y}\geq 0_{c_1},$$
and then
$$\eta(k_{c_{x,y}})> 0_{c_{\overline{1}}}=\eta(\omega_0) \text{ and }\zeta(k_{c_{x,y}})>(-1)_{c_{\overline{1}}}=\zeta(\omega_0),$$
which by Lemma \ref{lem:secprim} is equivalent to \eqref{eq:triangle}.
Conversely, if 
$$\eta(k_{c_{x,y}})> 0_{c_{\overline{1}}} \text{ and }\zeta(k_{c_{x,y}})>(-1)_{c_{\overline{1}}},$$
then 
$$\eta(k_{c_{x,y}})\geq 1_{c_1} \text{ and }\zeta(k_{c_{x,y}})\geq 0_{c_1}$$
so that $k\geq 1$. Thus $k\geq 1$ is equivalent to $\pi_{s-1}\gg_\rho \omega_0$.

\item For $i\in \{1,\ldots,n\}$, by  Lemma \ref{lem:secprim} we only need to show that \eqref{eq:star} is equivalent to
$$\eta(k_{c_{x,y}})>0_{c_{\overline{i+1}}}=\eta(\omega_i)\text{ and }\zeta(k_{c_{x,y}})> 0_{c_i}=\zeta(\omega_i).$$
\begin{itemize}
\item[$\bullet$] When we do not have $x\geq \overline{y}=i$, then by Remark \ref{rem:rk1}, we have
$k\geq H(b_i \otimes b_{x,y})=\rho(c_{x,y},c_{i,\overline{i}})$, and \eqref{eq:star} is  equivalent to $k_{c_{x,y}}\gg_\rho 0_{c_{i,\overline{i}}}$. By \eqref{eq:secprim}, the equation \eqref{eq:star} is also equivalent to
\begin{equation}
\label{eq:1}
\eta(k_{c_{x,y}})> 0_{c_{\overline{i}}} \text{ and }\zeta(k_{c_{x,y}})> 0_{c_i}.
\end{equation}
It suffices to show that \eqref{eq:1} is equivalent to
\begin{equation}
\label{eq:2}
\eta(k_{c_{x,y}})> 0_{c_{\overline{i+1}}} \text{ and }\zeta(k_{c_{x,y}})> 0_{c_i}.
\end{equation}
It is straightforward that \eqref{eq:1} implies \eqref{eq:2} as $0_{c_{\overline{i}}}>0_{c_{\overline{i+1}}}$. Suppose now that we have \eqref{eq:2}.
If we do not have\eqref{eq:1},
then
$$\eta(k_{c_{x,y}})= 0_{c_{\overline{i}}}\geq \zeta(k_{c_{x,y}})=0_{c_x} > 0_{c_i},$$
so that $y=\overline{i}\geq x>i$. Hence, we have $x\geq \overline{y}=i$ which is a contradiction.

\item[$\bullet$]When $x=\overline{y}=i$, \eqref{eq:star} becomes $k\geq 1$, and by \eqref{eq:k+1}, \eqref{eq:star} implies
$$\eta(k_{c_{i,\overline{i}}})\geq \eta(1_{c_{i,\overline{i}}})= 1_{c_i}>0_{c_{\overline{i+1}}}\text{ and }\zeta(k_{c_{i,\overline{i}}})\geq \eta(1_{c_{i,\overline{i}}})= 0_{c_{\overline{i}}} > 0_{c_i}.$$
Conversely, if 
$$\eta(k_{c_{i,\overline{i}}})>0_{c_{\overline{i+1}}}\text{ and }\zeta(k_{c_{i,\overline{i}}})> 0_{c_i},$$
then $\zeta(k_{c_{i,\overline{i}}})> \zeta(0_{c_{i,\overline{i}}})$ so that by \eqref{eq:k+1}, $k>0$.
\item[$\bullet$]Finally, when $x>\overline{y}=i$, \eqref{eq:star} becomes
$k\geq H(b_i \otimes b_{x,y})=0$, and by \eqref{eq:k+1}, is equivalent to
$$\eta(k_{c_{x,y}})\geq \eta(0_{c_{x,y}})=0_{c_{\overline{i}}}\text{ and }\zeta(k_{c_{x,y}})\geq \zeta(0_{c_{x,y}})= 0_{c_x} > 0_{c_i}.$$
The above relation is equivalent to saying that
$$\eta(k_{c_{x,y}})> 0_{c_{\overline{i+1}}} \text{ and }\zeta(k_{c_{x,y}})> 0_{c_i}.$$
\end{itemize}
\end{enumerate} 
\end{proof}

\section{Partitions described by frequencies on paths}
\label{sec:frequenciesonpaths}
So far we have worked with difference conditions, while the CMPP conjecture is expressed in terms of frequencies and paths. Our goal in this section is thus to connect these two different approaches.

We use the notation of Section \ref{sec:roadfrob} with $\Od$ an ordered set  with $m$ elements, where $m$ is a positive integer. Hence we identify $\Od$ with $\{1<\cdots<m\}$.

\subsection{Order and paths}

Let $\geq$ be the order on $\Z_\Sc$ with the following relation:
\begin{equation}\label{eq:Order}
k_c\geq l_d \text{ iff } \eta(k_c)>\eta(l_d) \text{ or } \Big(\eta(k_c)=\eta(l_d) \text{ and }\zeta(k_c)\leq\zeta(l_d) \Big).
\end{equation}

We first define operators on $\Z_\Rr$ and $\Z_\Sc$ that will help us formalise the notion of path.

\begin{deff}\label{def:succ}
Define the operator $succ$ on $\Z_\Rr$ such that $succ(k_c)$ is the smaller primary-coloured integer which is greater that $k_c$ in terms of $>$. We say that $succ(k_c)$ is the \textit{successor} of $k_c$. Precisely, we have 
$$
\begin{cases}
succ(k_{c_u})=k_{c_{u+1}}\ \ \text{for}\ \ u\in\{1,\ldots,m-1\},\\
succ(k_{c_m})=(k+1)_{c_1}.
\end{cases}
$$
Note that $succ$ is invertible, and $succ^{-1}(k_c)$, called the \textit{predecessor} of $k_c$, is the largest integer in $\Z_\Rr$ which is smaller than $k_c$.
\end{deff}
\begin{rem}
For all $k_c\in \Z_\Rr$, $succ^{m}(k_c)=(k+1)_c$.
\end{rem}
\begin{rem}
\label{rem:rk2}
By \eqref{eq:encadreta}, for all $k_c\in \Z_\Sc$, there exist a unique integer $u$ such that $\eta(k_c)=succ^u(\zeta(k_c))$. Moreover $u\in \{0,\ldots,m\}$.
\end{rem}

Now we want to define two operators $d$ and $f$ on $\Z_\Sc$ such that for $k_c \in \Z_\Sc$,
\begin{align*}
\eta(d(k_c))=\eta(k_c)&\text{ and }\zeta(d(k_c))=succ^{-1}(\zeta(k_c)),\\
\eta(f(k_c))=succ(\eta(k_c))&\text{ and }\zeta(f(k_c))=\zeta(k_c).
\end{align*}
In other words, we want $d$ to leave the $\eta$ component constant and transform the $\zeta$ component into its predecessor, and $f$ to leave the $\zeta$ component constant and transform the $\eta$ component into its successor.

By \eqref{eq:Order}, Definition \ref{def:succ} and Remark \ref{rem:rk2}, writing $\eta(k_c)=succ^u(\zeta(k_c))$, we have for all $k_c \in \Z_S$,
\begin{align}
\eta(f(k_c))=succ^{u+1}(\zeta(f(k_c)))&\text{ and }f(k_c)>k_c,\label{eq:deff}\\
\eta(d(k_c))=succ^{u+1}(\zeta(d(k_c)))&\text{ and }d(k_c)>k_c.\label{eq:defd}
\end{align}
Hence by Remark \ref{rem:rk2}, $d$ and $f$ can only be defined for $k_c$ such that $\eta(k_c)=succ^u(\zeta(k_c))$ with $u\in \{0,\ldots,m-1\}$, i.e. $\eta(k_c)<\zeta(k_c)+1$. Indeed if $u$ was equal to $m$, we would have $\eta(f(k_c))=succ^{m+1}(\zeta(f(k_c)))$ or $\eta(d(k_c))=succ^{u+1}(\zeta(d(k_c)))$, which is not possible.

Thus we define $d$ and $f$ as follows.
\begin{deff}\label{def:df}
Let $d$ and $f$ be two operators defined on $k_c \in \Z_\Sc$ such that $\eta(k_c)=succ^u(\zeta(k_c))$ with $u\in \{0,\ldots,m-1\}$, such that
\begin{align*}
\eta(d(k_c))=\eta(k_c)&\text{ and }\zeta(d(k_c))=succ^{-1}(\zeta(k_c)),\\
\eta(f(k_c))=succ(\eta(k_c))&\text{ and }\zeta(f(k_c))=\zeta(k_c).
\end{align*}
\end{deff}

\medskip
Now we introduce the notion of paths in $\Z_\Sc$ and show key properties satisfied by these paths.

\begin{deff}
\label{def:path}
A \textit{path} in $\Z_\Sc$ is a sequence $(e_0,\ldots,e_{m})$ of $m+1$ elements of $\Z_\Sc$ such that for all $u\in \{0,\ldots,m-1\}$, either $e_{u+1}=f(e_u)$ or $e_{u+1}=d(e_u)$.
\end{deff}

\begin{lem}\label{lem:path}
Let $(e_0,\ldots,e_{m})$ be a path in $\Z_\Sc$. Then it satisfies the following properties:
\begin{enumerate}
\item In terms of the order $\geq$, the sequence $(\eta(e_j))_{j=0}^m$ is non-decreasing, the sequence $(\zeta(e_j))_{j=0}^m$ is non-increasing, and the sequence $(e_j)_{j=0}^m$ is increasing.
\item For all $u\in \{0,\ldots,m\}$, $\eta(e_u)=succ^u(\zeta(e_u))$.
\end{enumerate}
\end{lem}


\begin{proof}
Property $(1)$ follows from Definitions \ref{def:succ} and \ref{def:df} and Relations \eqref{eq:deff} and \eqref{eq:defd}.

Now let us prove Property $(2)$. For a path $(e_0,\ldots,e_{m})$ such that $\eta(e_0)=succ^u(\zeta(e_0))$, we recursively have by \eqref{eq:deff} and \eqref{eq:defd} that $\eta(e_j)=succ^{u+j}(\zeta(e_j))$ for $j\in \{0,\ldots,m\}$, so that $0\leq u$ and $ u+m\leq m$, i.e. $u=0$. 

\end{proof}

We can now state a key result which characterises the relation $\gg_\rho$ in terms of paths in $\Z_\Sc$.
\begin{lem}\label{lem:orderpath}
For all $k_c \geq l_d\in \Z_\Sc$, we have $k_c\gg_\rho l_d$ if and only if there is no path in $\Z_\Sc$ which contains both $k_c$ and $l_d$.
\end{lem}

\begin{proof}
Let $k_c \geq l_d\in \Z_\Sc$. Showing the equivalence above is equivalent to showing that $k_c\not\gg_\rho l_d$ if and only if there is a path which contains both $k_c$ and $l_d$.
\begin{itemize}
\item[$\Rightarrow$)] If $k_c\not\gg_\rho l_d$, by \eqref{eq:secprim} and \eqref{eq:Order}, then either $\Big(\eta(k_c)>\eta(l_d)$ and $\zeta(k_c)\leq \zeta(l_d)\Big)$, or $\Big(\eta(k_c)=\eta(l_d)$ and $\zeta(k_c)\leq \zeta(l_d)\Big)$. In both cases, we have
$$\zeta(k_c)\leq \zeta(l_d) \leq \eta(l_d)\leq \eta(k_c) \leq \zeta(k_c)+1\,.$$
Hence, there exist $0\leq u\leq v\leq w\leq m$ such that 
$$\zeta(l_d) = succ^u(\zeta(k_c)),\, \eta(l_d)=succ^v(\zeta(k_c))\text{ and } \eta(k_c)=succ^w(\zeta(k_c))\,.$$
Let $e_0 \in \Z_\Sc$ be such that $\eta(e_0)=\zeta(e_0)=\zeta(l_d)$, and for all $t \in \{1, \dots, m \}$, let
$$
e_t :=
\begin{cases}
f^t(e_0) &\text{ if } t\in \{1,\ldots,v-u\},\\
d^{t-v+u}(f^{v-u}(e_0)) &\text{ if } t\in \{v-u+1,\ldots,v\},\\ 
f^{t-v}(d^{u}(f^{v-u}(e_0))) &\text{ if } t\in \{v+1,\ldots,m\}.
\end{cases}
$$
Hence $(e_0,\ldots,e_{m})$ is a path in $\Z_\Sc$.

Moreover, we have
\begin{itemize}
\item $e_{v-u}=l_d,$
\item $e_v$ is such that $(\eta(e_v),\zeta(e_v))=(\eta(l_d),\zeta(k_c))$,
\item $e_w=k_c$.
\end{itemize}
Therefore, $k_c$ and $l_d$ both belong to the path $(e_0,\ldots,e_{m})$.

\item[$\Leftarrow$)] Conversely, suppose that $k_c$ and $l_d$ belong to a path $(e_0,\ldots,e_{m})$ in $\Z_\Sc$. As $k_c \geq l_d$, we either have $\eta(k_c)>\eta(l_d)$, or $\eta(k_c)=\eta(l_d)$ and $\zeta(k_c)\leq \zeta(l_d)$.

\begin{itemize}
\item If $\eta(k_c)>\eta(l_d)$, we know that $(\eta(e_u))_{u=0}^{m}$ is non-decreasing by Lemma \ref{lem:path}. Thus there exist $u<v$ such that $l_d=e_u$ and $k_c=e_v$. Again, by Lemma \ref{lem:path}, $(\zeta(e_u))_{u=0}^{m}$ is non-increasing, and we have $$\zeta(k_c)=\zeta(e_v)\leq \zeta(e_u)=\zeta(l_d)\,.$$
Therefore, by \eqref{eq:secprim}, $k_c\not\gg_\rho l_d$.

\item If $\eta(k_c)=\eta(l_d)$, then by \eqref{eq:secprim}, $k_c\not\gg_\rho l_d$.

\end{itemize}

\end{itemize}

\end{proof}

\subsection{Frequencies, paths and partitions}\label{sec:nsproof}
In this section, we connect generalised coloured partitions with order $\gg_\rho$ to partitions with frequency conditions on paths in $\Z_\Sc$.

For $x\in \{1,\ldots,m\}$, set $\overline{x}:=m+1-x$. Also set $\omega_0:=(-1)_{c_{m,m}}$, and for $i\in \{1,\ldots,\lfloor m/2 \rfloor\}$,
$\omega_i:=0_{c_{i,\overline{i+1}}}$. Let 
$$\Omega = \{\omega_u:u\in \{0,\ldots,\lfloor m/2 \rfloor\}\}\sqcup \{0_{c_{u,\overline{u}}}: u\in \{1,\ldots,\lceil m/2 \rceil\}\}.$$
When $\mathcal{O}= \ov{[n]}$, the bars and $\omega_i$'s are defined as in the previous sections.

The sequence $\Ee=(e_0,\ldots,e_m)$ defined by $e_{m-2u}:=\omega_{u}$ for $u\in \{0,\ldots,\lfloor m/2 \rfloor\}$ and $e_{m-2u+1}:=0_{c_{u,\overline{u}}}$ for $u\in \{1,\ldots,\lceil m/2 \rceil\}$ is a path in $\Z_\Sc$. Indeed $e_{m-2u}=d(e_{m-2u-1})$ for $u\in \{0,\ldots,\lceil m/2 \rceil-1\}$ and $e_{m-2u+1}=f(e_{m-2u})$ for $u\in \{1,\ldots,\lfloor m/2 \rfloor\}$. Finally, set  
$$\Z_\Sc^+=\{0_{c_{x,y}}: m\geq y\geq x>\overline{y}\geq 1\}\sqcup(\Z_{>0})_\Sc,$$

Let $\Pp_\Sc$ denote the set of generalised coloured partitions with parts in $\Z_\Sc^+$ and order $\geq$. Recall that any partition $\pi\in \Pp_\Sc$ can be written as its frequency sequence $(f_u)_{u\in \Z_\Sc^+}$, where $f_u$ is the number of occurrences of $u$ in $\pi$.   We then have
$$C(\pi)q^{|\pi|}=\prod_{u\in \Z_\Sc^+} \left(c(u)q^{|u|}\right)^{f_u}.$$

The main result of this section is the following.

\begin{theo}\label{theo:mainns}
Let $\omega\in \Omega$.
Let $\Pp_\Sc^\omega$ be the set of partitions of $\Pp_\Sc$ whose frequency sequence $(f_u)_{u\in \Z_\Sc^+}$ is such that, by considering fictitious occurrences of elements in $\Omega$ with $f_{\omega}=1$, we have
$$f_{e_0}+\ldots+f_{e_{m}}\leq 1$$
for all paths $(e_0,\ldots,e_{m})$ in $\Omega\sqcup\Z_\Sc^+$.\\
Recall that $\Pp^{\omega}_\rho$ denotes the set of generalised coloured partitions with parts in $\Z_\Sc$, order $\gg_\rho$, and last part equal to $\omega$. 

There exists a bijection $\Lambda$ between $\Pp_\rho^{\omega}$  and $\Pp_\Sc^\omega$ which preserves the size and the colour sequence when omitting the last part $\omega$.
\end{theo}

To prove Theorem \ref{theo:mainns}, we first need two lemmas.
\begin{lem}\label{lem:omegaz+} 
Let $l_d\in \Omega\sqcup\Z_\Sc^+$ and let $k_c\in \Z_\Sc$ such that $\eta(k_c)\geq \eta(l_d)$ and $\zeta(k_c)\geq \zeta(l_d)$. Then
\begin{enumerate}
\item if $l_d\in \Z_\Sc^+$, then $k_c\in \Z_\Sc^+$,
\item if $l_d\in \Omega$, then $k_c\in \Omega\sqcup\Z_\Sc^+$. In particular, if  $k_c\gg_\rho l_d$, then $k_c\in \Z_\Sc^+$.
\end{enumerate}
\end{lem}

\begin{proof}
Let $k_c,l_d\in \Omega\sqcup\Z_\Sc^+$ such that $\eta(k_c)\geq \eta(l_d)$ and $\zeta(k_c)\geq \zeta(l_d)$ (and thus $k \geq l$). Write $c=c_{x,y}$ and $d=c_{x',y'}$. Observe that when $k\geq 1$, the result is straightforward as $k_c \in \Z_\Sc^+$.

\begin{enumerate}

\item Suppose that $l_d\in \Z_\Sc^+$. If $l>0$, then $k\geq l>0$ and we conclude.\\
Otherwise, $l=0$, $1\leq \overline{y}'<x'\leq m$  and $k\geq l\geq 0$. The case $k>0$ being obvious, suppose that $k=0$. By \eqref{eq:etazeta}, $0_{c_y}=\eta(k_c)\geq \eta(l_d)=0_{c_{y'}}$ and $0_{c_x}=\zeta_{k_c}\geq \zeta(l_d)=0_{c_{x'}}$ so that, by \eqref{eq:order},
$1\leq \overline{y}\leq \overline{y}'<x'\leq x\leq m$. Therefore, $k_c\in \Z_\Sc^+$.

\item Suppose that $l_d\in \Omega$.
 
\begin{enumerate}

\item When $l_d=\omega_0$, we have $\eta(k_c)\geq 0_{c_m}$ and $\zeta(k_c)\geq (-1)_{c_m}$.

\begin{itemize}
\item If $\zeta(k_c)= (-1)_{c_m}$, then 
$\eta(k_c)\leq \zeta(k_c)+1= 0_{c_m}$ so that $\eta(k_c)=0_{c_m}$ and $k_c=\omega_0 \in \Omega$.
\item If $\zeta(k_c)>(-1)_{c_m}$, then $\zeta(k_c)\geq 0_{c_1}$, and $k\geq 0$. The case $k>0$ being obvious, suppose that $k=0$. By \eqref{eq:etazeta}, $\zeta(k_c)=0_{c_y}\leq 0_{c_m}$, so that $\zeta(k_c)=0_{c_m}$ and $1=\overline{m}=\overline{y}\leq x$. If $x>1$, then $k_c \in \Z_\Sc^+$, else, $k_c=0_{c_{1,\overline{1}}}\in \Omega$. 
\end{itemize}

In particular, when $k_c\gg_\rho l_d=\omega_0$, then by \eqref{eq:secprim}, $\eta(k_c)> 0_{c_m}$ and $\zeta(k_c)> (-1)_{c_m}$ so that $\eta(k_c)\geq 1_{c_1}$ and $\zeta(k_c) \geq 0_{c_1}$. Hence, $k\geq 1$ and $k_c \in \Z_\Sc^+$.

\item When $l_d\neq\omega_0$, $l_d$ is of the form $0_{c_{i,\overline{i}}}$ or $0_{c_{i,\overline{i+1}}}$. Hence, by \eqref{eq:etazeta}, $\zeta(k_c)\geq \zeta(l_d)= 0_{c_{i}}$ and $\eta(k_c)\geq \eta(l_d)\geq 0_{c_{\overline{i+1}}}$, and thus $k\geq 0$. The case $k>0$ being obvious, suppose that $k=0$. Then, by \eqref{eq:etazeta}, $0_{c_{x}}\geq 0_{c_{i}}$ and $0_{c_{y}}\geq 0_{c_{\overline{i+1}}}$, i.e.  $x\geq i$ and $y\geq \overline{i+1}$.
\begin{itemize}
\item If $y=\overline{i+1}$ and $i\leq x\leq i+1$, then $k_c\in \Omega$. 
\item If $y=\overline{i+1}$ and $x>i+1$, then $\overline{y}<x$ and
$k_c \in \Z_\Sc^+$.
\item If $y=\overline{i}$ and $x=i$, then $k_c\in \Omega$. 
\item If $y=\overline{i}$ and $x>i$, then $\overline{y}<x$ and 
$k_c \in \Z_\Sc^+$.
\item If $y>\overline{i}$, then $\overline{y}<i\leq x$ and $k_c \in \Z_\Sc^+$.
\end{itemize}

In particular, when $k=0$ and $k_c\gg_\rho l_d$, by \eqref{eq:secprim} and \eqref{eq:etazeta}, $0_{c_x}=\zeta(k_c)\geq \zeta(l_d)=0_{c_{i}}$ and $0_{c_y}=\eta(k_c)>\eta(l_d)\geq 0_{c_{\overline{i+1}}}$. Hence, $\overline{y}<i+1\leq x$ and $k_c\in \Z_\Sc^+$.
\end{enumerate}
\end{enumerate}
\end{proof}

\begin{lem}\label{lem:inz+} 
Let $k_c,l_d\in \Omega\sqcup\Z_\Sc^+$ such that $k_c$ and $l_d$ belong to a same path in $\Z_\Sc$. Then, there exists a path in $\Omega\sqcup\Z_\Sc^+$ which contains both $k_c$ and $l_d$. 
\end{lem}

\begin{proof}[Proof of Lemma \ref{lem:inz+}]
Without loss of generality, suppose that $k_c\geq l_d \in \Omega\sqcup\Z_\Sc^+$ that $k_c$ and $l_d$ belong to a same path in $\Z_\Sc$. Similarly to the proof of Lemma \ref{lem:orderpath}, we have 
$$ \eta(k_c)-1 \leq \zeta(k_c)\leq \zeta(l_d) \leq \eta(l_d)\leq \eta(k_c).$$
Set $0\leq u\leq v\leq w \leq m$ such that
$$\eta(k_c) = succ^u(\eta(l_c))= succ^v(\zeta(l_d)) =succ^w(\zeta(k_c)).$$
Let $(e_0,\ldots,e_m)$ be the path defined by 
$$
\begin{cases}
(\eta(e_j),\zeta(e_j))=(\eta(l_d),succ^{-j}(\eta(l_d)))\ \ \text{for} \ j\in \{0,\ldots,v-u\},\\
(\eta(e_j),\zeta(e_j))=(succ^{j-v+u}(\eta(l_d)),\zeta(l_d)) \ \ \text{for} \ j\in \{v-u+1,\ldots,v\},\\
(\eta(e_j),\zeta(e_j))=(\eta(k_c),succ^{v-j}(\zeta(l_d)))) \ \ \text{for} \ j\in \{v+1,\ldots,w\},\\
(\eta(e_j),\zeta(e_j))=(succ^{j-w}(\eta(k_c)),\zeta(k_c)) \ \ \text{for} \ j\in \{w+1,\ldots,m\}.
\end{cases}
$$
Then, $e_{v-u}=l_d$ and $e_w = k_c$. For all $j\in \{0,\ldots,v\}$, $\eta(e_j)\geq \eta(l_d)$ and $\zeta(e_j)\geq \zeta(l_d)$, and Lemma \ref{lem:omegaz+} implies that $e_j\in \Omega\sqcup\Z_\Sc^+$. Similarly, for all $j\in \{v+1,\ldots,m\}$, $\eta(e_j)\geq \eta(k_c)$ and $\zeta(e_j)\geq \zeta(k_c)$ so that $e_j\in \Omega\sqcup\Z_\Sc^+$.
Hence $(e_0,\ldots,e_m)$ is a path in $\Omega\sqcup\Z_\Sc^+$ which contains both $k_c$ and $l_d$. 
\end{proof}

We are now ready to prove Theorem \ref{theo:mainns}.
\begin{proof}[Proof of Theorem \ref{theo:mainns}]
Let $\pi=(\pi_0,\ldots,\pi_{s-1},\omega)\in \Pp_\rho^{\omega}$. By \eqref{eq:secprim}, the relation $\gg_\rho$ is transitive. Thus, since $\pi_0\gg_\rho \cdots\gg_\rho\pi_{s-1}\gg_\rho \omega$, we have $\pi_u\gg_\rho \pi_v$ for all $0\leq u< v\leq s-1$ and $\pi_u\gg_\rho \omega$ for all $0\leq u\leq s-1$.
Hence, by (2) of Lemma \ref{lem:omegaz+}, $\pi_0,\ldots,\pi_{s-1}\in \Z_\Sc^+$. Since by \eqref{eq:secprim} and \eqref{eq:Order}, $k_c\gg_\rho l_d$ implies that $k_c> l_d$, then $\tilde{\pi}:=(\pi_0,\ldots,\pi_{s-1})$ belongs to  $\Pp_\Sc$.

Set $\Lambda(\pi):=\tilde{\pi}$.
The frequencies of $\tilde{\pi}$ are $f_u= 1$ for all $u\in \{\pi_0,\ldots,\pi_{s-1}\}$ and $f_u=0$ for all $u\in \Z_\Sc^+\setminus \{\pi_0,\ldots,\pi_{s-1}\}$. Finally, add fictitious frequencies $(f_u)_{u\in \Omega}$ with $f_{u}=\chi(u=\omega)$ for all $u\in \Omega$.

By Lemma \ref{lem:orderpath}, 
any path $(e_0,\ldots,e_{m})$ in $\Omega\sqcup\Z_\Sc^+ \subset \Z_\Sc$ contains at most one element of $\{\pi_0,\ldots,\pi_{s-1},\omega\}$, and then 
$$f_{e_0}+\ldots+f_{e_{m}}\leq 1.$$ 
Thus $\Lambda(\pi) \in \Pp_\Sc^\omega$.

\medskip

Conversely, let $\tilde{\pi}=(\pi_0,\ldots,\pi_{s-1})\in \Pp_\Sc^\omega$  and let $(f_u)_{u\in \Z_\Sc^+}$ be its frequency sequence. 
Then $f_u=0$ for all $u\in \Omega\setminus\{\omega\}$ as $\Ee$ is a path in $\Omega$. Moreover, $f_u\leq 1$ for all $u\in \Z_\Sc^+$, since there exists a path in $\Omega\sqcup\Z_\Sc^+$ containing $u$ by Lemma \ref{lem:inz+} in the case $u=k_c=l_d$. Hence,  $\pi_0>\cdots > \pi_{s-1}$. By Lemma \ref{lem:inz+}, two distinct elements in $\{\pi_0,\ldots, \pi_{s-1},\omega\}$ do not belong to the same path in $\Z_\Sc$, otherwise they would belong to a path in $\Omega\sqcup\Z_\Sc^+$. Therefore, by Lemma \ref{lem:orderpath}, for all $u\in \{0,\ldots,s-2\}$, we have $\pi_u\gg_\rho \pi_{u+1}$. Finally, 
by Lemma \ref{lem:orderpath}, we have either $\pi_{s-1}\gg_\rho \omega$ or $\omega\gg_\rho \pi_{s-1}$ according to whether $\pi_{s-1}>\omega$ or $\pi_{s-1}<\omega$. But, by Lemma \ref{lem:omegaz+}, as $\pi_{s-1}\in \Z_\Sc^+$, $\omega\not\gg_\rho \pi_{s-1}$. Therefore, $\pi_{s-1}\gg_\rho \omega$. The sequence $\pi=(\pi_0,\ldots,\pi_{s-1},\omega)$ is then well-ordered by $\gg_\rho$ and belongs to
$\Pp_\rho^\omega$. We set $\Lambda^{-1}(\tilde{\pi}):=\pi$. 

It is then straightforward that $\Lambda$ and $\Lambda^{-1}$ are inverses of each other as we only add/delete a part $\omega$.
The size and colour sequence are preserved by omitting the part $\omega$.
\end{proof}

\subsection{Specialisation}
The CMPP conjecture expresses the \textit{principally specialised} characters of standard modules of $C_n^{(1)}$ as a generating function for partitions. So far the present paper has only dealt with \textit{non-dilated} character formulas and partitions (which is more general). To make the connection with the conjecture, in this section, we give a dilated version of Theorem \ref{theo:mainns} which corresponds to the principal specialisation.

Recall that $(\Od,\geq)$ be an ordered set which can be identified with $\{1, \dots, m\}$ for some integer $m$,  that the set of primary colours is $\Rr=\{c_u: u\in \Od\}$, and that the set of secondary colours is $\Sc=\{c_{x,y}=c_xc_y=c_yc_x:x\leq y \in \Od\}.$ 

Let $\mathds{1}$ denote the transformation defined, for all $k_{c_j}\in \Z_\Rr$, by
$$\mathds{1}(k_{c_j})=km-\frac{m+1}{2}+j.$$
We have the relation $\mathds{1}(succ(k_c))=\mathds{1}(k_c)+1$ for all $k_c\in \Z_\Rr$.
 Therefore, the map $\mathds{1}$ describes a bijection from $\Z_\Rr$ to $\Z_m$, where
 $$\Z_m = \begin{cases}
 \Z  &\text{ if $m$ odd}\\
 \Z + \frac{1}{2} &\text{ if $m$ even}.
 \end{cases} $$

For all $k_c\in \Z_\Sc$, let $\mathds{1}(k_c)$ be the part of size $\mathds{1}(\eta(k_c))+\mathds{1}(\zeta(k_c))$ with subscript $\mathds{1}(\eta(k_c))-\mathds{1}(\zeta(k_c))$. By \eqref{eq:encadreta}, the subscript belongs to $\{0,\ldots,m\}$. For example, 
\begin{align*}
\mathds{1}((2k)_{c_{1,1}})=((2k-1)m+1)_0&\text{ and }\mathds{1}((2k+1)_{c_{1,1}})=(2km+1)_{m}\,,\\
\mathds{1}((2k)_{c_{1,m}})=(2km)_{m-1}&\text{ and }\mathds{1}((2k+1)_{c_{1,m}})=((2k+1)m)_{1}\,,\\
\mathds{1}((2k)_{c_{m,m}})=((2k+1)m-1)_0&\text{ and }\mathds{1}((2k-1)_{c_{m,m}})=(2km-1)_{m}\,.
\end{align*}
Observe that the difference between the size and subscript of $\mathds{1}(k_c)$ has the same parity as $m+1$. Conversely, for $l,d\in \Z$ such that $d\in \{0,\ldots,m\}$ and $l-d \equiv m+1\mod 2$, one can find a unique $k_c\in \Z_\Sc$ such that $\mathds{1}(k_c)=l_d$. Indeed, $\mathds{1}(\eta(k_c))=\frac{l+d}{2}$ and $\mathds{1}(\zeta(k_c))=\frac{l-d}{2}$ belong to $\Z_m$,
and $\eta(k_c),\zeta(k_c)$ are uniquely determined in $\Z_\Rr$. In the following, set for $k\in \Z\cup \{-\infty\}$,
$$E_k:=\{l_d: l\geq k\,,d\in\{0,\ldots,m\}\,, l-d\equiv m+1 \mod 2\}\,.$$  
Hence, the map $\mathds{1}$ is a bijection from $\Z_\Sc$ to $E_{-\infty}$, and the following statement holds.

\begin{lem}\label{lem:psz+}
We have
\begin{enumerate}
\item $\mathds{1}\left(\Z_\Sc^+\right)=E_1$,
\item $\mathds{1}(0_{c_{j,\overline{j}}}) = 0_{m-2j+1}$ for $j\in \{1,\ldots,\lceil m/2 \rceil\}$,
\item $\mathds{1}(\omega_j) = (-1)_{m-2j}$ for $j\in \{0,\ldots,\lfloor m/2 \rfloor\}$.
\end{enumerate}
\end{lem}

\begin{proof}
Note that for all $k_c\in \Z_\Sc$, $\mathds{1}(k_c+2)=\mathds{1}(k_c)+2m$, since by \eqref{eq:succession}, $\mathds{1}(\eta(k_c+2))=\mathds{1}(\eta(k_c)+1)=\mathds{1}(\eta(k_c))+m$, and $\mathds{1}(\zeta(k_c+2))=\mathds{1}(\zeta(k_c)+1)=\mathds{1}(\zeta(k_c))+m$. Hence, writing
$$Z:=\{0_{c_{x,y}}: 1\leq \overline{y}< x\leq y\leq m\}\sqcup\{1_c:c\in \Sc\}\sqcup \{2_{c_{x,y}}: 1\leq x\leq y< \overline{x}\leq m\},$$
we have
$$\Z_\Sc^+ = \{z + 2k : k \geq 0 , z \in Z\} \sqcup \{O_{c_{j\ov{j}}} + 2k : k >0 , j \in \{1, \dots , \lceil m/2 \rceil \}\}.$$
Thus, to prove (1) and (2), it suffices to show that $\mathds{1}\left(Z\right)= E_{1}\setminus E_{2m}$  and $\mathds{1}(0_{c_{j,\overline{j}}}) = 0_{m-2j+1}$ for $j\in \{1,\ldots,\lceil m/2 \rceil\}$.

For $i\in \{1,\ldots,\lceil m/2 \rceil\}$, we have
\begin{align*}
\mathds{1}(\zeta(0_{c_{j,\overline{j}}}))&=\mathds{1}(0_{c_j})=-\frac{m+1}{2}+j,\\
\mathds{1}(\eta(0_{c_{j,\overline{j}}}))&=\mathds{1}(0_{c_{\overline{j}}})=\frac{m+1}{2}-j,\\
\end{align*}
and then  $\mathds{1}(0_{c_{j,\overline{j}}}) = 0_{m+1-2j}$. For $1\leq \overline{y}< x\leq y\leq m$, 
\begin{align*}
\mathds{1}(\eta(0_{c_{x,y}}))&=\mathds{1}(0_{c_y})\leq \mathds{1}(0_{c_m})= \frac{m-1}{2} ,\\
\mathds{1}(\zeta(0_{c_{x,y}}))&=\mathds{1}(0_{c_x})> \mathds{1}(0_{c_{\overline{y}}})=-\mathds{1}(0_{c_{y}}),
\end{align*}
and then 
$0<\mathds{1}(\eta(0_{c_{x,y}}))+\mathds{1}(\eta(0_{c_{x,y}}))< m$.  For $1\leq x\leq y< \overline{x}\leq m$, 
\begin{align*}
\mathds{1}(\zeta(2_{c_{x,y}}))&=\mathds{1}(1_{c_x})\geq \mathds{1}(1_{c_{1}})= \frac{m+1}{2} ,\\
\mathds{1}(\eta(2_{c_{x,y}}))&=\mathds{1}(1_{c_y})<\mathds{1}(1_{c_{\overline{x}}})= m+\mathds{1}(0_{c_{\overline{x}}}) =m-\mathds{1}(0_{c_{x}})=2m-\mathds{1}(1_{c_{x}}),
\end{align*}
and then $m<\mathds{1}(\eta(0_{c_{x,y}}))+\mathds{1}(\eta(0_{c_{x,y}}))< 2m$. Finally, for $1\leq x\leq y\leq m$,
\begin{align*}
m+\frac{m-1}{2}=\mathds{1}(1_{c_{m}})\geq \mathds{1}(\eta(1_{c_{x,y}}))&=\mathds{1}(1_{c_x})\geq \mathds{1}(1_{c_{1}})= m+\frac{1-m}{2},\\
\frac{m-1}{2}=\mathds{1}(0_{c_{m}})\geq \mathds{1}(\zeta(1_{c_{x,y}}))&=\mathds{1}(0_{c_y})\geq \mathds{1}(0_{c_{1}})= \frac{1-m}{2},
\end{align*}
and then $0<\mathds{1}(\eta(0_{c_{x,y}}))+\mathds{1}(\eta(0_{c_{x,y}}))< 2m$. Hence, $\mathds{1}\left(Z\right)\subset E_1\setminus E_{2m}$.

Moreover, 
as 
\begin{align*}
\Sc &= \{c_{x,y}:1\leq x\leq y\leq m\}
\\&=\{c_{x,y}:1\leq x\leq y< \overline{x}\leq m\}\sqcup \{c_{x,y}:1\leq x\leq y=\overline{x}\leq m\}\sqcup \{c_{x,y}:1\leq \overline{y}< x\leq y\leq m\},
\end{align*}
we have
$|Z|=2|\Sc|-\lceil m/2\rceil= m(m+1)-\lceil m/2\rceil$. Furthermore,  in $E_1\setminus E_{2m}$, the $m$ odd numbers appear with subscripts of the same parity as $m$, and the $m-1$ even numbers appear with subscripts of the same parity as $m+1$.
Thus
$$|E_1\setminus E_{2m}|=m(1+\lfloor m/2\rfloor)+ (m-1)\lceil m/2\rceil = m(1+\lfloor m/2\rfloor+\lceil m/2\rceil)-\lceil m/2\rceil=m(m+1)-\lceil m/2\rceil.$$
Along with the fact that $\mathds{1}$ is injective, we obtain that $\mathds{1}\left(Z\right)= E_1\setminus E_{2m}$. Thus (1) and (2) are proved.

Let us now prove (3). We compute $\mathds{1}(\omega_i)$ for $i\in \{0,\ldots,\lfloor m/2\rfloor\}$. For $i=0$,  
$$\mathds{1}(\eta(\omega_0))=\mathds{1}(0_{c_m})=\frac{m-1}{2}\text{ and }\mathds{1}(\zeta(\omega_0))=\mathds{1}((-1)_{c_{m}})=-\frac{m+1}{2},$$
and then $\mathds{1}(\omega_0)=(-1)_{m}$. For $1\leq j\leq \lfloor m/2\rfloor$, 
$$\mathds{1}(\eta(\omega_j))=\mathds{1}(0_{c_{\overline{j+1}}})=-j+\frac{m-1}{2}\text{ and }\mathds{1}(\zeta(\omega_j))=\mathds{1}(0_{c_j})=-\frac{m+1}{2}+j,$$
and then $\mathds{1}(\omega_j)=(-1)_{m-2j}$.
\end{proof}

By Lemma \ref{lem:psz+}, $\mathds{1}(\Omega)=E_{-1}\setminus E_1$.
If $(e_0,\ldots,e_{m})$ is a path in $\Z_\Sc$, then $\mathds{1}(e_j)$ is of the form $(l^{(j)})_j$ with 
$$l^{(j+1)}=\begin{cases}
l^{(j)}+1 &\text{ if } e_{j+1}=f(e_j),
\\ l^{(j)}-1 &\text{ if } e_{j+1}=d(e_j).
\end{cases}$$
A sequence $((l^{(0)})_0, \dots , (l^{(m)})_m)$ such that for all $j \in \{0, \dots , m-1\}$, $l^{(j+1)}= l^{(j)} \pm 1$  and $(l^{(j)})_j \in E_k$ ($k \in \Z \cup \{- \infty \}$) is called \textit{a path in $E_k$}.

Conversely, any path in $E_{-\infty}$ is the image by $\mathds{1}$ of a path in $\Z_\Sc$. By Theorem \ref{theo:mainns} and Lemma \ref{lem:psz+}, we derive the following theorem.

\begin{theo}[Dilated version of Theorem \ref{theo:mainns}]
\label{theo:mainspbis}
Let $\omega\in \Omega$.
Denote by $\Pp^{\mathds{1}(\omega)}$ the set of partitions with parts in $E_1$ such that, letting $f_u$ be the frequency of $u$ for all $u\in E_1$, and setting fictitious frequencies $f_u:=\chi(u=\mathds{1}(\omega))$ for all $u\in E_{-1}\setminus E_1$, we have
$$f_{e_0}+\cdots+f_{e_{m}}\leq 1$$
for all paths $(e_0,\ldots,e_{m})$ in $E_{-1}$. Then,
$$\sum_{\pi\in \Pp^{\mathds{1}(\omega)}}q^{|\pi|}= \mathds{1}\left(c(\omega)^{-1}q^{-|\omega|}\sum_{\pi\in \Pp_\rho^{\omega}}C(\pi)q^{|\pi|}\right).$$
\end{theo}
\begin{rem}
In the formula above, multiplying $\sum_{\pi\in \Pp_\rho^{\omega}}C(\pi)q^{|\pi|}$ with the factor $c(\omega)^{-1}q^{-|\omega|}$ amounts to not counting the fictitious part $\omega$ in the generating function.
\end{rem}

\subsection{Proof of the CMPP conjecture of level one weights and generalisation}
\label{sec:proofCMPP}
Finally, we use particular cases of the results of the previous sections to prove Theorem \ref{theo:mainsp}, i.e. that the CMMP conjecture in the case of level $1$ standard modules $L(\Lambda_0), \dots, L(\Lambda_n)$ of $C_n^{(1)}$.

Let $m=2n$. One can then identify $\Od$ with the set $ \ov{[n]}=\{1,\ldots,n,\overline{n},\ldots,\overline{1}\}$ associated to the crystal $\B$ of $C_n^{(1)}$, and the parts $\omega_i$ defined in Section \ref{sec:nsproof} are exactly those of Section \ref{sec:minimal}.
By setting
$$
c_{\overline{j}}^{-1}=c_j= e^{\frac{1}{2}\alpha_n+\sum_{u=j}^{n-1}\alpha_u} \text{ for } j\in \{1,\ldots,n\},
$$
where the $\alpha_i$'s are the simple roots,
we have $c_{x,y} = c_xc_y = e^{\wt{(x,y)}}$ for all $(x,y)\in \B\setminus \{\emptyset\}$. The principal specialisation is the transformation $e^{-\alpha_j}\mapsto q$, and since $\delta = \alpha_0+\alpha_n+2\sum_{j=1}^{n-1}\alpha_j$, we have that $e^{-\delta}\mapsto q^{2n}$. Hence, the principal specialisation corresponds exactly to $\mathds{1}$. 
Thus Theorem \ref{theo:mainspbis}, together with the identification of $\Pp_{i,\rho}$ to $\Pp_\rho^{\omega_i}$ in Lemma \ref{lem:identification}, implies Theorem \ref{theo:mainsp}.

\section{Conclusion}
In \cite{CMPP}, Capparelli--Meurman--Primc--Primc gave an analogue of their conjecture for odd moduli, which would correspond in our setting to the case $m=2n-1$.
We reformulate this conjecture as follows.

\begin{con}[Reformulation of the CMPP odd conjecture]
Let $k_0,\ldots,k_n$ be non-negative integers and denote by $\mathcal{Q}^{k_0,\ldots,k_n}_{\mathds{1}}$ the set of partitions into parts in $E_1$ such that, by setting $f_u$ the frequency of $u\in E_1$ and fictitious occurrences of $u\in E_{-1}\setminus E_1$ with $f_{(-1)_{2n-1-2i}}=k_i$ for $i\in \{0,\ldots,n-1\}$ and $f_{(0)_{0}}=k_n$,
$$f_{e_0}+\cdots+f_{e_{2n-1}}\leq k_0+\cdots+k_n$$
for all path $(e_0,\ldots,e_{2n-1})$ of $E_{-1}$. Then,
$$\sum_{\pi\in \mathcal{Q}^{k_0,\ldots,k_n}_{\mathds{1}}}q^{|\pi|}= \frac{\prod_{a\in \{2n+2k+1\}^n;b\in \Delta(k_1+1,\ldots,k_n+1); j=a,b,2n+2k+1-b}(q^j;q^{2n+2k+1})_\infty}{(q;q)_\infty^n},
$$
where $k=k_0+\cdots+k_n$.
\end{con}
In the case $n=1$, we retrieve the Andrews-Gordon identities, which, in the literature, are linked to the representations of type $A_1^{(1)}$. It is then legitimate to look for some Lie-algebraic interpretation of the general conjecture. The non-specialised version of the conjecture would then be in the following form.

\begin{con}[Non-specialized version of the conjecture]
Let $k_0,\ldots,k_n$ be non-negative integers, and denote by $\mathcal{Q}^{k_0,\ldots,k_n}_\Sc$ the set of partitions into parts in $\Z_\Sc^+$ such that, by setting $(f_u)_{u\in \Z_\Sc^+}$ the frequencies of the parts and considering fictitious occurrences of elements in $\Omega$ with $f_{\omega_i}=k_i$ for $i\in\{0,\ldots,n-1\}$ and $f_{0_{c_{n,n}}}=k_n$, we have 
$$f_{e_0}+\ldots+f_{e_{2n-1}}\leq k_0+\ldots+k_n$$
for all path $(e_0,\ldots,e_{2n-1})$ of $\Omega\sqcup\Z_\Sc^+$. 
Hence, for some suitable variables $q$ and $c_i$ for all $i\in \{1,\ldots,n\}$, we have 
\begin{equation}
\sum_{\pi\in \mathcal{Q}^{k_0,\ldots,k_n}_\Sc} C(\pi)q^{|\pi|} = e^{-k_0\Lambda_0-\cdots-k_n\Lambda_n}\ch(M(k_0\Lambda_0+\cdots+k_n\Lambda_n))
\end{equation}
for a certain highest weight module $M(k_0\Lambda_0+\cdots+k_n\Lambda_n)$ of a certain type $T$. 
\end{con}

In the case of $A_n^{(1)}$ which we treated in \cite{DK19}, we were also able to rewrite our generating functions for coloured Frobenius partitions as the constant term in an infinite product, and as a sum of infinite products. This yielded two non-specialised character formulas with obviously positive coefficients and allowed us to retrieve the Kac-Peterson \cite{KacPeterson} formula. Unfortunately we have not yet been able to find an equally simple formula for the generating function of $C_n^{(1)}$-Frobenius partitions. Progress in that direction would be interesting as it would lead to non-specialised character formulas with obviously positive coefficients for level $1$ standard modules of $C_n^{(1)}$ as well.

\section*{Acknowledgements}
The first author is supported by the ANR COMBIN\'e ANR-19-CE48-0011 and the SNSF Eccellenza grant number PCEFP2\_202784. The second author is funded by the LABEX MILYON (ANR-10-LABX-0070) of Universit\'e de Lyon, within the program ``Investissements d'Avenir" (ANR-11-IDEX-0007) operated by the French National Research Agency (ANR).

\bibliographystyle{alpha}
\bibliography{biblio.bib}

\end{document}